\documentclass[preprint,12pt]{elsarticle}

\usepackage[top=2.5cm, bottom=2cm, left=3cm, right=3cm]{geometry}  
\usepackage[english]{babel}
\usepackage{amsmath,amsthm,amsfonts,amssymb,mathrsfs}
\usepackage{mathtools}
\usepackage{hyperref}

\newtheorem{theorem}{Theorem}[section]
\newtheorem{definition}[theorem]{Definition}
\newtheorem{example}[theorem]{Example}

\newtheorem{lemma}[theorem]{Lemma}
\newtheorem{corollary}[theorem]{Corollary}
\newtheorem{proposition}[theorem]{Proposition}

\DeclarePairedDelimiter\ceil{\lceil}{\rceil}
\DeclarePairedDelimiter\floor{\lfloor}{\rfloor}
\def\lub{\mathbf{lub}}

\begin{document}

\begin{frontmatter}

%% Title, authors and addresses

%% use the tnoteref command within \title for footnotes;
%% use the tnotetext command for theassociated footnote;
%% use the fnref command within \author or \affiliation for footnotes;
%% use the fntext command for theassociated footnote;
%% use the corref command within \author for corresponding author footnotes;
%% use the cortext command for theassociated footnote;
%% use the ead command for the email address,
%% and the form \ead[url] for the home page:
%% \title{Title\tnoteref{label1}}
%% \tnotetext[label1]{}
%% \author{Name\corref{cor1}\fnref{label2}}
%% \ead{email address}
%% \ead[url]{home page}
%% \fntext[label2]{}
%% \cortext[cor1]{}
%% \affiliation{organization={},
%%             addressline={},
%%             city={},
%%             postcode={},
%%             state={},
%%             country={}}
%% \fntext[label3]{}

\title{Weierstrass semigroups at totally ramified places of degree one on linearized function fields}

%% use optional labels to link authors explicitly to addresses:
%% \author[label1,label2]{}
%% \affiliation[label1]{organization={},
%%             addressline={},
%%             city={},
%%             postcode={},
%%             state={},
%%             country={}}
%%
%% \affiliation[label2]{organization={},
%%             addressline={},
%%             city={},
%%             postcode={},
%%             state={},
%%             country={}}

\author[1]{Huachao Zhang}
\ead{zhanghch56@mail2.sysu.edu.cn}

\author[1,2]{Chang-An Zhao\corref{cor1}}
\ead{zhaochan3@mail.sysu.edu.cn}
\cortext[cor1]{Corresponding author.}
%% Author affiliation
\affiliation[1]{organization={School of Mathematics, Sun Yat-sen University},%Department and Organization
            city={Guangzhou},
            postcode={510275}, 
            state={Guangdong},
            country={China}}
\affiliation[2]{organization={Guangdong Key Laboratory of Information Security Technology},%Department and Organization
            city={Guangzhou},
            postcode={510006}, 
            state={Guangdong},
            country={China}}  
\begin{abstract}
    A linearized function field $F$ can be viewed as a Galois extension of a rational function field $K(x)$. For a totally ramified place $Q$ of degree one in $F/K(x)$, we give a unified description of the set $G(Q)$ of gaps at $Q$.
	As a consequence, we explicitly provide a system of generators, the multiplicity, and the Frobenius number of the Weierstrass semigroup $H(Q)$. Moreover, we give a necessary and sufficient condition for $H(Q)$ to be symmetric. 
	Then we investigate the minimal generating set of the Weierstrass semigroups at several totally ramified places of degree one.
	We not only explicitly describe the minimal generating set, but also provide functions whose coefficients of pole divisors lie in the minimal generating set.
    Finally, we investigate the linearized function field associated with the denominator of a separable polynomial and apply our results to present several examples.
\end{abstract}

\begin{keyword}
Algebraic function fields, Linearized function fields, Gaps, Weierstrass semigroups
%% PACS codes here, in the form: \PACS code \sep code
\MSC[2020] 14H55 \sep 11R58 
%% or \MSC[2008] code \sep code (2000 is the default)
\end{keyword}

\end{frontmatter}

\section{Introduction}
	Let $K$ be an algebraic extension of the finite field $\mathbb{F}_q$ with $q$ elements, where $q$ is a power of a prime $p$.
	Consider a function field $F/K$ with full constant field $K$ and genus $g\geq 1$. Given a rational place $P$ of $F$, 
	the Weierstrass semigroup at $P$, commonly written as $H(P)$, is a classical object in algebraic geometry; a number of recent developments concerning it  can be found in the literature, including \cite{castellanosWeierstrassSemigroupsPure2024,mendozaKummerExtensions2023,abdonWeierstrassPoints2019,beelenWeierstrassSemigroups2021,bartoliWeierstrassSemigroups2021a,beelenWeierstrassSemigroups2018}.
	The complementary set $G(P) := \mathbb{N}_0\setminus H(P)$ is called the set of gaps at $P$. The largest element in $G(P)$, denoted by $F_{H(P)}$, is called the Frobenius number of $H(P)$. 
   	The Weierstrass semigroup $H(P)$ is said to be symmetric if $F_{H(P)} = 2g-1$. Recently, several studies have investigated the symmetry of Weierstrass semigroups for specific function fields; see, e.g., \cite{cotterillGapSets2025,mendozaKummerExtensions2023,guneriAutomorphismGroup2013}.

	The Weierstrass semigroup and the set of gaps at a rational place have significant applications in coding theory and algebraic geometry. In particular, one can apply them to obtain algebraic geometry codes with good parameters; see, for instance, \cite{castellanosOneTwoPointCodes2016,castellanosWeierstrassSemigroupsPure2024,montanucciAGCodes2020,garciaConsecutiveWeierstrass1993,korchmarosHermitianCodes2013a}.
    They can bound the number of rational points on a curve \cite{geilBoundingNumber2009}.
	They also play a key role in classifying maximal function fields, as explored in \cite{niemannNonisomorphicMaximal2025,beelenFamilyNonisomorphic2025}.
	Moreover, they serve as a powerful tool for identifying the automorphism group of a given algebraic curve or function field, as illustrated in \cite{maAutomorphismGroups2016a,montanucciAutomorphismGroup2024,beelenWeierstrassSemigroups2023,beelenWeierstrassSemigroups2026a,beelenWeierstrassSemigroups2026}.

	Let $P_1,\dots,P_s$ be $s$ distinct rational places of the function field $F$.
	The Weierstrass semigroup $H(P_1,\dots,P_s)$ at several rational places has attracted considerable attention and been extensively investigated over the past decades.
	The theory of Weierstrass semigroups at two rational places was originally developed by Arbarello et al. in \cite{arbarelloGeometryAlgebraic1985}.
	Their work was later extended by Kim \cite{kimIndexWeierstrass1994} and Homma \cite{hommaWeierstrassSemigroup1996}, while the general case involving several rational places was first addressed by Carvalho and Torres \cite{carvalhoGoppaCodes2005}.
	In \cite{matthewsWeierstrassSemigroup2004}, Matthews introduced the concept of a minimal generating set for $H(P_1,\dots,P_s)$, denoted by $\tilde{\Gamma}(P_1,\dots,P_s)$, which is fully determined by certain special sets $\Gamma(P_1,\dots,P_k)$ for $1\leq k\leq s$.
	Subsequently, Matthews computed $\Gamma(P_1,\dots,P_k)$ for $1\leq k\leq s$ on a quotient of the Hermitian curve \cite{matthewsWeierstrassSemigroups2005}, as well as on norm-trace curves \cite{matthewsWeierstrassSemigroups2009,matthewsMinimalGenerating2010}. In \cite{matthewsTriplesRational2021}, Matthews, Skabelund and Wills gave an explicit description of the Weierstrass semigroup for each triple of places on the Hermitian curve.
	Extensive research has also been devoted to Weierstrass semigroups at several rational places on other algebraic curves and function fields.
	For example, studies have focused on GK curves \cite{tizziottiWeierstrassSemigroup2018}, curves defined by equations of the form $f(x) = g(x)$ \cite{castellanosWeierstrassSemigroup2018}, GGS curves \cite{huMultipointCodes2020}, maximal curves that cannot be covered by the Hermitian curve \cite{castellanosWeierstrassSemigroup2020}, Kummer extensions \cite{yangWeierstrassSemigroups2017,huMultipointCodes2018,castellanosGeneralizedWeierstrass2026}, and the third function field in a tower attaining the Drinfeld-Vl\u adu\c t bound \cite{yangWeierstrassSemigroups2024}.

	Let $L(y) = \sum_{i=0}^n \alpha_i y^{p^i}\in K[y]$ be a separable linearized polynomial with $\alpha_0,\alpha_n\neq 0$, having all its $p^n$ roots in $K$.
	Let $f(x), g(x)\in K[x]$ be two coprime polynomials. Consider the linearized function field $F = K(x,y)$ defined by 
	$$L(y) = \frac{f(x)}{g(x)}.$$
	Under certain conditions, the extension $F/K(x)$ is a Galois extension of function fields. In the case where $\deg f(x) > \deg g(x)$, Navarro determined an explicit basis for the Riemann-Roch space $\mathcal{L}(D)$ for a divisor $D$ consisting of some ramified places in the extension $F/K(x)$, and then determined the dimension of $\mathcal{L}(D)$ \cite{navarroBasesRiemann2024}.
	
	In this work, we investigate the linearized function field in the general case; that is, we do not require $\deg f(x)>\deg g(x)$. Inspired by main theorems in \cite{navarroBasesRiemann2024} and techniques in \cite{maharajCodeConstruction2004}, we determine the dimension of $\mathcal{L}(D)$ for a divisor $D$ consisting of some ramified places in $F/K(x)$ (see Corollary \ref{D_dim}).
	Let $Q$ be a totally ramified place of degree one in $F/K(x)$. Then we give an explicit description of the set of gaps $G(Q)$ in a unified way (see Proposition \ref{gap_set}).  
	As a consequence,  we explicitly provide a system of generators, the multiplicity, and the Frobenius number of the Weierstrass semigroup $H(Q)$ (see Corollaries \ref{semigroup} and \ref{m_and_F}).
	Furthermore, we obtain a necessary and sufficient condition for $H(Q)$ to be symmetric (see Theorem \ref{sym_thm}).
	Let $Q_1,\dots,Q_s$ be $s$ distinct totally ramified places of degree one in $F/K(x)$. Following the ideas in \cite{matthewsWeierstrassSemigroup2004,matthewsWeierstrassSemigroups2005,yangWeierstrassSemigroups2017}, we fully determine the set $\Gamma(Q_1,\dots,Q_s)$.
	We not only explicitly describe the set $\Gamma(Q_1,\dots,Q_s)$, but also provide functions whose coefficients of pole divisors lie in $\Gamma(Q_1,\dots,Q_s)$ (see Theorem \ref{thm_minimal_set}).
	Finally, let $g(x)$ be a separable polynomial. We study the linearized function field in two cases: one where $\deg f(x)\leq \deg g(x)$ and the other where $\deg f(x)>\deg g(x)$.

	This paper is organized as follows. In Section \ref{section2}, we briefly recall some notations and preliminary results related to Weierstrass semigroups and linearized function fields.
	In Section \ref{section3}, for a totally ramified place $Q$ of degree one in $F/K(x)$, we explicitly describe the set $G(Q)$ of gaps and the Weierstrass semigroup $H(Q)$,
	and then we provide a necessary and sufficient condition for $H(Q)$ to be symmetric. In Section \ref{section4}, we focus on computing $\Gamma(Q_1,\dots,Q_s)$ for $s$ distinct totally ramified places of degree one in $F/K(x)$.
	In Section \ref{section5}, we investigate the linearized function field associated with the denominator of a separable polynomial in detail and exhibit some examples by using our results.

\section{Preliminaries}\label{section2}
    Throughout this article, let $q$ be a power of a prime $p$, and let $\mathbb{F}_q$ be the finite field with $q$ elements. Let $K$ be an algebraic extension of $\mathbb{F}_q$. 
    For a set $A$, we denote by $|A|$ the cardinality of $A$. If $A$ is an infinite set, we adopt the convention that $c<|A|$ for every $c\in\mathbb{Z}$.
	For $a,b\in\mathbb{Z}$, we denote by $\gcd(a,b)$ the greatest common divisor of $a$ and $b$. Let $\mathbb{N} = \{1,2,3,\dots\}$ and $\mathbb{N}_0 = \{0,1,2,\dots\}$.
    For $c\in \mathbb{R}$, we denote by $\floor{c}$ the largest integer not greater than $c$ and by $\ceil{c}$ the smallest integer not less than $c$.

	\begin{lemma}\cite[Lemma 4.1]{castellanosWeierstrassSemigroups2024}\label{ceil_and_floor}
	Let $a$ and $b$ be elements in $\mathbb{R}$. The following statements hold:

	(i) $\floor*{-a} = -\ceil*{a}$.
	
	(ii) $\ceil*{a}-\floor*{a} = \begin{cases} 0, & \text{if } a \in \mathbb{Z}, \\ 1, & \text{if } a \notin \mathbb{Z}. \end{cases}$
	
	(iii) If $a$ and $b$ are positive integers, then
	$$\sum_{k=1}^{b-1} \left\lfloor \frac{ka}{b} \right\rfloor = \frac{(a-1)(b-1) + \gcd(a,b) - 1}{2}.$$
	\end{lemma}

\subsection{Function fields}
Let $F/K$ be a function field with constant field $K$. Let $g$ be the genus of $F$. We denote by $\mathbf{P}_F$ the set of places of $F$, and by $\operatorname{Div} (F)$ the free abelian group generated by the places in $F$. For each place $P\in\mathbf{P}_F$, we denote by $v_P$ the discrete valuation and by $\mathcal{O}_P$ the valuation ring with respect to $P$. A place of degree one is called a rational place of $F$. An element $D\in \operatorname{Div}(F)$ is called a divisor of $F$ and its degree is given by $\deg D := \sum_{P\in{\operatorname{supp} D}} v_P(D)\cdot \deg P$, where $\operatorname{supp} D$ is the support of $D$.
For a non-zero element $z\in F$, we denote by $(z)_F$, $(z)_\infty$ and $(z)_0$ the principal divisor, the pole divisor and the zero divisor of $z$, respectively.

Given a divisor $D\in \operatorname{Div}(F)$, we define the Riemann-Roch space associated to $D$:  
$$\mathcal{L}(D) := \{z\in F\mid (z)_F \geq -D\}\cup\{0\}.$$
We have that $\mathcal{L}(D)$ is a vector space over $K$, and denote by $\ell(D)$ the dimension of $\mathcal{L}(D)$. 
Let $W$ be a canonical divisor of $F/K$. Then for each divisor $D\in \operatorname{Div}(F)$, the Riemann-Roch Theorem says that 
$$\ell(D) = \deg D + 1 - g + \ell(W-D).$$

Next we introduce the notion of Weierstrass semigroups in $F$. 
Let $P_1,\dots,P_s$ be $s$ rational places of $F$. The Weierstrass semigroup at $P_1,\dots,P_s$ is defined by 
$$H(P_1,\dots,P_s) := \left\{(a_1,\dots,a_s)\in \mathbb{N}_0^s~|~\exists z\in F \text{~with~}(z)_\infty = \sum_{i=1}^s a_iP_i\right\}.$$
The complementary set $G(P_1,P_2,\dots,P_s):= \mathbb{N}_0^s \setminus H(P_1,\dots, P_s)$ is called the set of gaps at $P_1,P_2,\dots,P_s$. An element in $G(P_1,P_2,\dots,P_s)$ is called a gap at $P_1,P_2,\dots,P_s$.

Let $P$ be a rational place of $F$. The Weierstrass semigroup $H(P)$ at a single rational place has more properties. If $g \geq 1$, then $G(P) = \mathbb{N}_0\setminus H(P)$ has exactly $g$ elements $1 = i_1<i_2<\dots <i_g\leq 2g-1$ at $P$.
The smallest nonzero element of $H(P)$ is called the multiplicity of $H(P)$ and is denoted by $m_{H(P)}$.
The largest element of $G(P)$ is called the Frobenius number and is denoted by $F_{H(P)}$. We say that the Weierstrass semigroup $H(P)$ is symmetric if $F_{H(P)} = 2g-1$.

\subsection{The minimal generating sets of Weierstrass semigroups}

To describe the minimal generating sets of Weierstrass semigroups, we first introduce additional notation. For two elements $\mathbf{a} = (a_1,\dots,a_s), \mathbf{b} = (b_1,\dots,b_s)\in \mathbb{N}_0^s$, we define a partial order $\preceq$ on $\mathbb{N}_0^s$ by $\mathbf{a}\preceq \mathbf{b}$ if and only if $a_i\leq b_i$ for all $1\leq i\leq s$.
Furthermore, if $a_i<b_i$ for some $1\leq i\leq s$, we write $\mathbf{a}\prec \mathbf{b}$. Let $S\subseteq \mathbb{N}_0^s$ and $\mathbf{a} \in S$. We say that $\mathbf{a}$ is minimal in $S$ with respect to $\preceq$ if $\mathbf{b}\not\preceq \mathbf{a}$ for all $\mathbf{b}\in S\setminus\{\mathbf{a}\}$.

For a function field $F/K$ with genus $g>0$, let $P_1,\dots,P_s$ be $s$ distinct rational places of $F$, where $s\leq |K|$.
As shown by Carvalho and Torres \cite{carvalhoGoppaCodes2005}, the dimensions of Riemann-Roch spaces can be used to characterize $H(P_1,P_2,\dots,P_s)$ and $G(P_1,P_2,\dots,P_s)$. Given an $s$-tuple $\mathbf{a} = (a_1,\dots,a_s)\in \mathbb{N}_0^s$, we have that $\mathbf{a}\in H(P_1,\dots,P_s)$ if and only if 
$$\ell\left(\sum_{i=1}^s a_iP_i\right) = \ell\left(\sum_{i=1}^s a_iP_i-P_j\right)+1 \text{ for all } 1\leq j\leq s.$$
Moreover, we have that $\mathbf{a}\in G(P_1,\dots,P_s)$ if and only if 
$$\ell\left(\sum_{i=1}^s a_iP_i\right) = \ell\left(\sum_{i=1}^s a_iP_i-P_j\right) \text{ for some } 1\leq j\leq s.$$

Next, we introduce the definition of $\Gamma(P_1,\dots,P_t)$ for $1\leq t\leq s$, which is proposed by Matthews \cite{matthewsWeierstrassSemigroup2004}. Let $\Gamma(P_1) := H(P_1)$. For $s\geq 2$, define $\Gamma(P_1,\dots,P_s)$ by
$$\{\mathbf{a}\in \mathbb{N}_0^s~|~ \mathbf{a} \text{ is minimal in }\{\mathbf{b}\in H(P_1,\dots,P_s)~|~b_i = a_i\} \text{~for some~} 1\leq i\leq s\}.$$
\begin{proposition}\cite[Proposition 3]{matthewsWeierstrassSemigroup2004}\label{minimal_for_all}
	Let $\mathbf{a} = (a_1,\dots,a_s)\in\mathbb{N}_0^s$. Then $\mathbf{a}$ is minimal in $\{\mathbf{b}\in H(P_1,\dots,P_s)~|~b_i = a_i\}$ with respect to $\preceq$ for some $1\leq i\leq s$,
	if and only if $\mathbf{a}$ is minimal in $\{\mathbf{b}\in H(P_1,\dots,P_s)~|~b_i = a_i\}$ with respect to $\preceq$ for all $1\leq i\leq s$.
\end{proposition}
\begin{lemma}{\cite[Lemma 2.6]{castellanosWeierstrassSemigroup2018}}\label{discrepancy}
	Let $\mathbf{a} = (a_1, \dots, a_s) \in H(P_1, \dots, P_s)$ and $A = a_1P_1 + \dots +a_sP_s$. Then $\mathbf{a} \in \Gamma(P_1, \dots, P_s)$ if and only if 
    $$\ell(A)= \ell(A - P) +1 = \ell(A - P - Q) +1 \text{ and } \ell(A) = \ell(A - Q) + 1 = \ell(A - P - Q)+1$$
	for any two rational places $ P, Q \in \{P_1, \dots, P_s\} $.
\end{lemma}

\begin{lemma}{\cite[Lemma 4]{matthewsWeierstrassSemigroup2004}}\label{Gamma_cross}
	Suppose that $s\geq 2$. Then 
	$$\Gamma(P_1,\dots,P_s) \subseteq G(P_1)\times\dots\times G(P_s).$$
\end{lemma}

For $s=2$, suppose that $G(P_1) = \{a_1<a_2<\dots<a_g\}$ and $G(P_2) = \{b_1<b_2<\dots<b_g\}$.
For each gap $a_i$ at $P_1$, let $n_{a_i} = \min\{b\in\mathbb{N}_0\mid(a_i,b)\in H(P_1,P_2)\}$.
From \cite[Lemma 2.6]{kimIndexWeierstrass1994}, we have the equality $\{n_a \mid a \in G(P_1)\} = G(P_2)$, and therefore there exists a permutation $\tau$ of $\{1, 2, \dots, g\}$ such that $n_{a_i} = b_{\tau(i)}$.
The graph of the bijection between $G(P_1)$ and $G(P_2)$ defining the permutation $\tau$ is the set $\Gamma(P_1, P_2) = \{(a_i, b_{\tau(i)}) \mid i = 1, \dots, g\}$.
The following lemma characterizes it.
\begin{lemma}\cite[Lemma 2]{hommaWeierstrassSemigroup1996}\label{minimal_generating_two}
	Let $\Gamma$ be a subset of $(G(P_1) \times G(P_2)) \cap H(P_1, P_2)$. 
	If there exists a permutation $\tau$ of $\{1, 2, \dots, g\}$ such that $\Gamma = \{(a_i, b_{\tau(i)}) \mid i = 1, \dots, g\}$, then $\Gamma = \Gamma(P_1, P_2)$.
\end{lemma}

Let $1\leq t\leq s$ and $I = \{i_1,\dots,i_t\}\subseteq \{1,\dots,s\}$. Define the natural inclusion
$$\begin{array}{ccccc}
	\iota_I:& \mathbb{N}_0^{t} & \longrightarrow & \mathbb{N}_0^s,\\
	&(a_{i_1},\dots,a_{i_t}) & \longmapsto &(a_1,\dots,a_s),
\end{array}$$
where $a_{j} = 0$ for $j\not\in I$, and define the natural projection
$$\begin{array}{ccccc}
	\pi_I:& \mathbb{N}_0^{s} & \longrightarrow & \mathbb{N}_0^t,\\
	&(a_1,\dots,a_s) & \longmapsto &(a_{i_1},\dots,a_{i_t}).
\end{array}$$
The minimal generating set of $H(P_1,\dots,P_s)$ is defined as
$$\tilde{\Gamma}(P_1,\dots,P_s) := \bigcup_{t=1}^s\bigcup_{\substack{I = \{i_1,\dots,i_t\}\\ 1\leq i_1<\dots<i_t\leq s}} \iota_I(\Gamma(P_{i_1},\dots,P_{i_t})).$$
Given $\mathbf{u_1},\dots,\mathbf{u_t}\in \mathbb{N}_0^s$, where $t\geq 2$ and $s\geq 1$, define the least upper bound of $\mathbf{u_1},\dots,\mathbf{u_t}$ by
$$\lub\{\mathbf{u_1},\dots,\mathbf{u_t}\} := (\max\{u_{1_1},\dots,u_{t_1}\},\dots,\max\{u_{1_s},\dots,u_{t_s}\}).$$
The following theorem shows that $H(P_1,\dots,P_s)$ is determined by $\tilde{\Gamma}(P_1,\dots,P_s)$.
\begin{theorem}\cite[Theorem 7]{matthewsWeierstrassSemigroup2004}\label{lub_semi_group}
	Suppose that $s\geq 2$. Then 
	$$H(P_1,\dots,P_s) = \{\lub\{\mathbf{u_1},\dots,\mathbf{u_s}\} ~|~\mathbf{u_1},\dots,\mathbf{u_s}\in \tilde{\Gamma}(P_1,\dots,P_s)\}.$$
\end{theorem}

\subsection{Linearized function fields}
Let $L(y) = \sum_{i=0}^n \alpha_i y^{p^i}\in K[y]$ be a separable linearized polynomial with $\alpha_0,\alpha_n\neq 0$, having all its $p^n$ roots in $K$.
Let $f(x), g(x)\in K[x]$ be two coprime polynomials, and let $s,r\in \mathbb{N}_0$ with $(s,r)\neq (0,0)$. We consider the linearized function field $F = K(x,y)$ defined by the equation 
\begin{equation}\label{equation1}
	L(y) = \frac{f(x)}{g(x)} =  \alpha\cdot \frac{\prod_{j=0}^r q_j(x)^{m_j}}{\prod_{i=0}^{s}p_i(x)^{n_i}},
\end{equation}
where $\alpha\in K$, $n_1,\dots,n_s, m_1,\dots,m_r\in \mathbb{N}$, $p_1(x),\dots,p_s(x)$, $q_1(x),\dots,q_r(x)\in K[x]$ are pairwise distinct monic irreducible polynomials, $p_0(x) = q_0(x) = 1$, $n_0 = \sum_{j=1}^r m_j\cdot\deg q_j(x) - \sum_{i=1}^{s} n_i\cdot\deg p_i(x)$, and $m_0 = -n_0$. 
Moreover, suppose that $\gcd(n_i,p) = 1$ for all $1\leq i\leq s$, and suppose that either $n_0\leq 0$ or $n_0>0$ with $\gcd(n_0,p) = 1$. 
According to \cite[Proposition 3.7.10]{stichtenothAlgebraicFunctionFields2009}, the extension $F/K(x)$ is a Galois extension and has several useful properties.

For the convenience of describing the zeros and poles of $y$, we introduce two sets. We define
$$I := \left\{\begin{array}{l}\{1,2,\dots,s\}, \text{ if } n_0 \leq 0,\\
	\{0,1,\dots,s\}, \text{ if } n_0>0, \end{array}\right. \text { and }
J := \left\{\begin{array}{l}\{1,2,\dots,r\}, \text{ if } n_0 \geq 0,\\
	\{0,1,\dots,r\}, \text{ if } n_0<0. \end{array}\right.$$
For $1\leq i\leq s$, we denote by $P_i$ the zero of $p_i(x)$ in $\mathbf{P}_{K(x)}$. We denote by $P_\infty$ the pole of $x$ in $\mathbf{P}_{K(x)}$. For convenience, we also denote by $P_0$ the pole of $x$ in $\mathbf{P}_{K(x)}$.
Each place $P$ in $\{P_i\in\mathbf{P}_{K(x)}\mid i\in I\}$ is totally ramified in $F/K(x)$.
For each $i\in I$, we denote by $Q_i$ the only place in $\mathbf{P}_F$ lying over $P_i$. 

Let $d_i := \deg\,p_i(x)$ for $1\leq i\leq s$ and $e_j := \deg\,q_j(x)$ for $1\leq j\leq r$. Let $d_0 = e_0 := 1$. 
Let $m := \sum_{i\in I} n_id_i$. It is clear that $m = \sum_{j\in J} m_je_j$.
By the Hurwitz Genus Formula, the genus of $F$ is 
$$g = \frac{(p^n-1)(m+\sum_{i\in I} d_i-2)}{2}.$$
Let $\beta_1 = 0,\beta_2,\dots,\beta_{p^n}$ be all roots of $L(y)$.
For each $1\leq j\leq r$, we denote by $S_{j}$ the zero of $q_j(x)$ in $\mathbf{P}_{K(x)}$, and denote by $R_{jk}$ the only zero of $(y-\beta_k)$ in $\mathbf{P}_F$ lying over $S_j$ for $1\leq k\leq p^n$.
If $J = \{0,1,\dots,r\}$, we denote by $R_{0k}$ the only zero of $(y-\beta_k)$ in $\mathbf{P}_F$ lying over $P_0$.
Let $D_0 := {\rm Con}_{F/K(x)} (P_0) $, where ${\rm Con}_{F/K(x)}$ is the conorm with respect to $F/K(x)$.
We have the following divisors 
\begin{equation}\label{divisor1}
	(p_i(x))_F = p^n Q_i - d_iD_0 \text{ for } i\in I,
\end{equation}
\begin{equation}\label{divisor2}
	(q_j(x))_F =\sum_{k=1}^{p^n}R_{jk} - e_j D_0 \text{ for } j\in J,
\end{equation}
\begin{equation}\label{divisor3}
	(y-\beta_k)_F = \sum_{j\in J} m_j R_{jk} - \sum_{i\in I} n_i Q_i \text{ for } 1\leq k\leq p^n.
\end{equation}
For any divisor $D\in \operatorname{Div}(F)$, we write
$$D = \sum_{P\in \mathbf{P}_{K(x)}}\sum_{Q\in \mathbf{P}_{F}, Q\mid P} a_Q Q,$$
where $a_Q$ are integers such that $a_Q = 0$ for almost all $Q\in \mathbf{P}_{F}$.
We define the restriction of $D$ to $K(x)$ as
$$D|_{K(x)} := \sum_{P\in \mathbf{P}_{K(x)}} \min\left\{\floor*{\frac{a_Q}{e(Q | P)}}: Q|P\right\}P.$$ 
\begin{lemma}\cite[Lemma 2.1]{maharajCodeConstruction2004}\label{RR_cap}
	Using the notation above, we have
	$$\mathcal{L}(D)\cap K(x) = \mathcal{L}(D|_{K(x)}).$$
\end{lemma}

\section{The Weierstrass semigroup at a totally ramified place of degree one on linearized function fields}\label{section3}
In this section, we consider the linearized function field $F = K(x,y)$ defined by the equation \eqref{equation1}. 
For a totally ramified place $Q$ of degree one in $F/K(x)$, we provide a unified and explicit description of the set $G(Q)$ of gaps at $Q$.
We then determine a system of generators, the multiplicity, and the Frobenius number of the Weierstrass semigroup $H(Q)$.
In addition, we establish a necessary and sufficient condition for $H(Q)$ to be symmetric. 
We begin this section by presenting a theorem which is similar to \cite[Theorem 2.2]{maharajCodeConstruction2004}.

\begin{theorem}\label{LD_oplus}
	Let $D\in\operatorname{Div}(F)$ be a divisor with $\operatorname{supp}(D) \subseteq \{Q_i\mid i\in I\}$. Then 
	$$\mathcal{L}(D) = \bigoplus_{k=0}^{p^n-1} \mathcal{L}\left(\left(D+(y^k)_F\right)|_{K(x)}\right)y^k.$$
\end{theorem}
\begin{proof}
	Note that $\{1,y,\dots,y^{p^n-1}\}$ is a basis of $F/K(x)$.
	For $z_0,z_1,\dots,z_{p^n-1}\in K(x)$, we will show that 
	$$z_0 + z_1y + \dots + z_{p^n-1}y^{p^n-1} \in \mathcal{L}(D) \text{ if and only if } z_ky^k\in \mathcal{L}(D) \text{ for all } 0\leq k\leq p^n-1.$$
	The reverse implication is clear. Suppose that
	$$z = z_0 + z_1y + \dots + z_{p^n-1}y^{p^n-1} \in \mathcal{L}(D)$$
	is a non-zero element. We first show that the only possible poles in $\mathbf{P}_{K(x)}$ of $z_k$ are in $\{P_i\mid i\in I\}$ for all $0\leq k\leq p^n-1$.
	Indeed, assume that $P\notin \{P_i\mid i\in I\}$ is a pole of $z_k$ for some $0\leq k\leq p^n-1$.
	Note that for any place $Q\in \mathbf{P}_F$ lying over $P$, we have $Q\notin \operatorname{supp} (D)$, which implies $v_Q(z)\geq 0$ and then $z\in \bigcap_{Q|P} \mathcal{O}_Q$.
	It is clear that the minimal polynomial of $y$ over $K(x)$ satisfies
	$$\varphi(T) = \sum_{i=0}^{n} \alpha_i T^{p^i} - \alpha\cdot \frac{\prod_{j=0}^r q_j(x)^{m_j}}{\prod_{i=0}^{s}p_i(x)^{n_i}} \in \mathcal{O}_P[T],$$
	and $v_Q(\varphi'(y)) = v_Q(\alpha_0) = 0$. By \cite[Corollary 3.5.11]{stichtenothAlgebraicFunctionFields2009}, we have that $\{1,y,\dots,y^{p^n-1}\}$ is an integral basis of $F$ at $P$.
	Thus the integral closure $\mathcal{O}'_P$ of $\mathcal{O}_P$ in $F$ is
	$$\mathcal{O}'_P = \bigcap_{Q|P} \mathcal{O}_Q = \sum_{k=0}^{p^n-1}\mathcal{O}_P y^k.$$
	This implies that $z\in \mathcal{O}'_P$ and then each $z_k\in \mathcal{O}_P$, which is a contradiction. 
	Hence, we obtain that $v_{Q}(z_ky^k) = v_{Q}(z_k) + kv_Q(y) \geq 0 $ for all $Q\in \mathbf{P}_F \setminus \{Q|P_i : i\in I\}$.

	For each $i\in I$ and $0\leq k\leq p^n-1$, we have 
	$$v_{Q_i}(z_ky^k) = p^n v_{P_i}(z_k) + kv_{Q_i}(y) = p^nv_{P_i}(z_k) - kn_i.$$
	Since $p^n v_{P_i}(z_k) - kn_i$ are distinct modulo $p^n$, we get that $v_{Q_i}(z_ky^k)$ are distinct. Thus, we obtain that 
	$$v_{Q_i}(z_ky^k) \geq \min_{0\leq k\leq p^n-1}\{v_{Q_i}(z_ky^k)\} = v_{Q_i}(z) \geq -v_{Q_i}(D).$$
	Therefore, we have $z_k y^k \in \mathcal{L}(D)$.

	Now, by Lemma \ref{RR_cap}, for $0\leq k\leq p^n-1$, we have that $z_ky^k\in \mathcal{L}(D)$ if and only if 
	\begin{align*}
		z_k\in \left(\mathcal{L}(D)/y^k\right)\cap K(x) = \mathcal{L}\left(D + (y^k)_F\right)\cap K(x) = \mathcal{L}\left(\left(D+(y^k)_F\right)|_{K(x)}\right).
	\end{align*}
	The proof is completed.
\end{proof}

	Applying the above theorem, we can determine the dimension of $\mathcal{L}(D)$ by Riemann-Roch Theorem.
	The following result extends \cite[Theorem 3.4]{navarroBasesRiemann2024}. 
\begin{corollary}\label{D_dim}
	Suppose that $\{l_1,\dots,l_{|I|}\}$ be a permutation of $I$ and $1\leq t\leq |I|$. Let $(a_1,\dots,a_t)\in \mathbb{Z}^{t}$. Then 
	$$\ell\left(\sum_{i=1}^{t} a_i Q_{l_i}\right) = \sum_{k=0}^{p^n-1}\max\left\{0,~ \sum_{i=1}^{t} \floor*{\frac{a_i-kn_{l_i}}{p^n}}d_{l_i} + \sum_{i=t+1}^{|I|}\floor*{\frac{-kn_{l_i}}{p^n}}d_{l_i}+1\right\}.$$
\end{corollary}
\begin{proof}
	For each $0\leq k\leq p^n-1$, we have 
	$$\sum_{i=1}^{t} a_iQ_{l_i} + (y^k)_F = \sum_{i=1}^{t} (a_i - kn_{l_i})Q_{l_i} - \sum_{i=t+1}^{|I|}kn_{l_i} Q_{l_i} + \sum_{j\in J}km_j R_{j0},$$
	and the restriction of this divisor to $K(x)$ is 
	$$\left(\sum_{i=1}^{t} a_iQ_{l_i} + (y^k)_F \right)\Big|_{K(x)} = \sum_{i=1}^{t} \floor*{\frac{a_i-kn_{l_i}}{p^n}} P_{l_i} + \sum_{i=t+1}^{|I|}\floor*{\frac{-kn_{l_i}}{p^n}}P_{l_i}.$$
	By Theorem \ref{LD_oplus}, we have 
	$$\mathcal{L}\left(\sum_{i=1}^{t}a_iQ_{l_i}\right) = \bigoplus_{k=0}^{p^n-1} \mathcal{L}\left(\left(\sum_{i=1}^{t}a_i Q_{l_i}+(y^k)_F\right)\Big|_{K(x)}\right)y^k.$$
	Thus
	$$\ell\left(\sum_{i=1}^{t}a_iQ_{l_i}\right) = \sum_{k=0}^{p^n-1}\ell\left(\sum_{i=1}^{t} \floor*{\frac{a_i-kn_{l_i}}{p^n}} P_{l_i} + \sum_{i=t+1}^{|I|}\floor*{\frac{-kn_{l_i}}{p^n}}P_{l_i}\right).$$
	It follows from Riemann-Roch Theorem that 
	$$\ell\left(\sum_{i=1}^{t} a_i Q_{l_i}\right) = \sum_{k=0}^{p^n-1}\max\left\{0,~ \sum_{i=1}^{t} \floor*{\frac{a_i-kn_{l_i}}{p^n}}d_{l_i} + \sum_{i=t+1}^{|I|}\floor*{\frac{-kn_{l_i}}{p^n}}d_{l_i} + 1\right\}.$$
\end{proof}

The next lemma characterizes the set of gaps at a single totally ramified place of degree one in $F/K(x)$.
\begin{lemma}\label{one_gap}
	Suppose that $l\in I$ and $d_l = 1$.
	Then $a\in \mathbb{N}$ is a gap at $Q_l$ if and only if 
	$$\floor*{\frac{a-kn_l}{p^n}}+\sum_{i\in I,i\neq l}\floor*{\frac{-kn_i}{p^n}}d_{i}\leq -1,$$
	where $k\in\{0,1,\dots,p^n-1\}$ is the unique element such that $a-kn_l\equiv 0\pmod {p^n}$.
\end{lemma}
\begin{proof}
	Let $a\in \mathbb{N}$. We have that $a$ is a gap at $Q_l$ if and only if $\ell(aQ_l) = \ell((a-1)Q_l)$. For $k\in\{0,1,\dots,p^n-1\}$, note that $\floor*{\frac{a-kn_l}{p^n}}\neq \floor*{\frac{a-kn_l-1}{p^n}}$ if and only if $a-kn_l\equiv 0 \pmod {p^n}$.
	Then by Corollary \ref{D_dim}, we obtain that $a\in \mathbb{N}$ is a gap at $Q_l$ if and only if 
	$$\floor*{\frac{a-kn_l}{p^n}}+\sum_{i\in I,i\neq l}\floor*{\frac{-kn_i}{p^n}}d_i\leq -1,$$
	where $k\in\{0,1,\dots,p^n-1\}$ is the unique element such that $a-kn_l\equiv 0\pmod {p^n}$.
\end{proof}

Now we explicitly provide a unified description of the gap set at a single totally ramified place of degree one in $F/K(x)$.
\begin{proposition}\label{gap_set}
	Let $\lambda\in\mathbb{Z}$ with $\gcd(p,\lambda) = 1$.
	Suppose that $l\in I$ and $d_l = 1$. 
	Then 
	\begin{align*}
	G(Q_l) = \Bigg\{jp^n - i\lambda n_l ~\Big | ~ &1\leq i\leq p^n-1,\\
	&\ceil*{\frac{i\lambda n_l}{p^n}}\leq j\leq m\ceil*{\frac{i\lambda}{p^n}} - \sum_{k\in I}\floor*{\frac{i\lambda n_k}{p^n}}d_k + \floor*{\frac{i\lambda n_l}{p^n}}- 1\Bigg\}.
	\end{align*}
	In particular, suppose that $\lambda = -1$. Then
	\begin{equation}\label{description2}
		G(Q_l) = \left\{jp^n + in_l ~\Big | ~ 1\leq i\leq p^n-1,\, \ceil*{\frac{-in_l}{p^n}}\leq j\leq \sum_{k\in I}\ceil*{\frac{in_k}{p^n}}d_k - \ceil*{\frac{in_l}{p^n}}-1\right\}.
	\end{equation}

	Furthermore, suppose that $\lambda$ is the inverse of $-n_l$ modulo $p^n$. Then
	\begin{equation}\label{description3}
	G(Q_l) = \Bigg\{jp^n+i ~\Big | ~ 1\leq i\leq p^n-1,\, 0\leq j\leq m\ceil*{\frac{i\lambda}{p^n}} - \sum_{k\in I}\floor*{\frac{i\lambda n_k}{p^n}}d_k-2\Bigg\}.
	\end{equation}
\end{proposition}
\begin{proof}
	Define the set 
	\begin{align*}
	G = \Bigg\{jp^n - i\lambda n_l ~\Big | ~ &1\leq i\leq p^n-1,\\
	&\ceil*{\frac{i\lambda n_l}{p^n}}\leq j\leq m\ceil*{\frac{i\lambda}{p^n}} - \sum_{k\in I}\floor*{\frac{i\lambda n_k}{p^n}}d_k + \floor*{\frac{i\lambda n_l}{p^n}}- 1\Bigg\}.
	\end{align*}
	For $jp^n-i\lambda n_l\in G$, let $u\in\{0,1,\dots,p^n-1\}$ be the unique element such that $jp^n-i\lambda n_l-un_l\equiv 0\pmod {p^n}$.
	Since $\gcd(p,n_l) = 1$ we obtain $ u+i\lambda \equiv 0 \pmod {p^n}$. It follows that $-i\lambda = p^n\floor*{\frac{-i\lambda}{p^n}} + u$. Then 
	\begin{align*}
		~&\floor*{\frac{jp^n-i\lambda n_l-un_l}{p^n}}+\sum_{k\in I,k\neq l}\floor*{\frac{-un_k}{p^n}}d_k\\
		=~&j + n_l\floor*{\frac{-i\lambda}{p^n}} + \floor*{\frac{-i\lambda}{p^n}}\sum_{k\in I,k\neq l} n_k d_k + \sum_{k\in I,k\neq l} \floor*{\frac{i\lambda n_k}{p^n}}d_k \\
		=~& j - m \ceil*{\frac{i\lambda}{p^n}} + \sum_{k\in I}\floor*{\frac{i\lambda n_k}{p^n}}d_k-\floor*{\frac{i\lambda n_l}{p^n}}\leq -1.
	\end{align*}
	Applying Lemma \ref{one_gap}, we conclude that $jp^n-i\lambda n_l\in G(Q_l)$. This yields $G\subseteq G(Q_l)$.

	Let $\lambda = bp^n + \bar{\lambda}$, where $b\in\mathbb{Z}$ and $1\leq \bar{\lambda}\leq p^n-1$.
	Suppose that $j_1p^n - i_1\lambda n_l, j_2p^n-i_2\lambda n_l\in G$ with $j_1p^n - i_1\lambda n_l = j_2p^n-i_2\lambda n_l$. Then we have $(i_1-i_2)\lambda n_l \equiv 0 \pmod {p^n}$. Since $\gcd(p,n_l) = \gcd(p,\lambda)=1$ we obtain that $i_1 = i_2$ and then $j_1 = j_2$.
	Thus by Lemma \ref{ceil_and_floor}, we get
	{\allowdisplaybreaks\begin{align*}
		\# G &= \sum_{i=1}^{p^n-1} \left(m\ceil*{\frac{i\lambda}{p^n}}-\sum_{k\in I}\floor*{\frac{i\lambda n_k}{p^n}}d_k + \floor*{\frac{i\lambda n_l}{p^n}} - \ceil*{\frac{i\lambda n_l}{p^n}}\right)\\
		& = \sum_{i=1}^{p^n-1} \left(m-1+m\floor*{\frac{i\lambda}{p^n}}-\sum_{k\in I}\floor*{\frac{i\lambda n_k}{p^n}}d_k \right)\\
		& = (m-1)(p^n-1) + m \sum_{i=1}^{p^n-1} \floor*{\frac{i\lambda}{p^n}} - \sum_{k\in I}d_k\sum_{i=1}^{p^n-1}\floor*{\frac{i\lambda n_k}{p^n}}\\
		& = (m-1)(p^n-1) + m \sum_{i=1}^{p^n-1} \floor*{\frac{i\bar{\lambda}}{p^n}} - \sum_{k\in I}d_k\sum_{i=1}^{p^n-1}\floor*{\frac{i\bar{\lambda} n_k}{p^n}}\\
		& = (m-1)(p^n-1) + m\frac{(\bar{\lambda}-1)(p^n-1)}{2}- \sum_{k\in I}\frac{d_k(\bar{\lambda} n_k-1)(p^n-1)}{2}\\
		& = (m-1)(p^n-1) + \frac{(m\bar{\lambda}-m)(p^n-1)}{2}- \frac{(m\bar{\lambda}-\sum_{k\in I}d_k)(p^n-1)}{2}\\
		& = \frac{(p^n-1)(m+\sum_{k\in I}d_k-2)}{2} = g.
	\end{align*}}
	This concludes the desired result $G(Q_l) = G$.
	In particular, let $\lambda = -1$. For each $1\leq i\leq p^n-1$, note that $m\ceil*{\frac{-i}{p^n}} = 0$ and $\ceil*{\frac{-in_l}{p^n}} = \floor*{\frac{in_l}{p^n}}$.
	Thus the desired expression is obtained.

	Furthermore, suppose that $-\lambda n_l = 1 + bp^n$ for some $b\in\mathbb{Z}$. Then we have
	\begin{align*}
	G(Q_l) = \Bigg\{ j'p^n + i'+i'bp^n ~\Big | ~ &1\leq i'\leq p^n-1,\\
	&\ceil*{\frac{i'\lambda n_l}{p^n}}\leq j'\leq m\ceil*{\frac{i'\lambda}{p^n}} - \sum_{k\in I}\floor*{\frac{i'\lambda n_k}{p^n}}d_k + \floor*{\frac{i'\lambda n_l}{p^n}}- 1\Bigg\}.
	\end{align*}
	Note that $\ceil*{\frac{i'\lambda n_l}{p^n}} = -i'b  + \ceil*{\frac{-i'}{p^n}} = -i'b$ and $\floor*{\frac{i'\lambda n_l}{p^n}} = -i'b + \floor*{\frac{-i'}{p^n}} = -i'b-1$ for all $1\leq i'\leq p^n-1$.
	Let $i = i'$ and $j = j'+i'b$, then we obtain that 
	\begin{align*}
	G(Q_l) = \Bigg\{jp^n+i ~\Big | ~ &1\leq i\leq p^n-1,\, 0\leq j\leq m\ceil*{\frac{i\lambda}{p^n}} - \sum_{k\in I}\floor*{\frac{i\lambda n_k}{p^n}}d_k-2\Bigg\}.
	\end{align*}
\end{proof}

By the description \eqref{description3} of $G(Q_l)$, we immediately obtain the following corollary.
\begin{corollary}\label{semigroup_equal}
    Suppose that $d_i = d_j = 1$ for some $i,j\in I$. If $n_i \equiv n_j \pmod {p^n}$, then $G(Q_i) = G(Q_j)$ and $H(Q_i) = H(Q_j)$.
\end{corollary}
\begin{proof}
    Let $1\leq \lambda \leq p^n-1$ be inverse of $-n_i$ modulo $p^n$. Since $n_i \equiv n_j \pmod {p^n}$, we have that $\lambda$ is also the inverse of $-n_j$. It follows from Proposition \ref{gap_set} that $G(Q_i) = G(Q_j)$ and then $H(Q_i) = H(Q_j)$.
\end{proof}

Moreover, from the description \eqref{description2} of $G(Q_l)$, we will obtain some corollaries, which are able to determine a system of generators, the multiplicity, and the Frobenius number of the Weierstrass semigroup $H(Q_l)$.
\begin{corollary}\label{semigroup}
Suppose that $l\in I$ and $d_l = 1$. Then
$$H(Q_l) = \left\langle p^n, \left(\sum_{k\in I}\ceil*{\frac{i n_k}{p^n}}d_k-\ceil*{\frac{in_l}{p^n}}\right)p^n + in_l : 1\leq i\leq p^n-1 \right\rangle.$$
\end{corollary}
\begin{proof}
    It follows from \eqref{description2} that 
    \begin{align*}
        H(Q_l) = & ~ \mathbb{N}_0\setminus G(Q_l) \\
        = & ~\left\{ jp^n+in_l ~\Big|~ 1 \leq i \leq p^n-1, \, j\geq \sum_{k\in I}\ceil*{\frac{in_k}{p^n}}d_k - \ceil*{\frac{in_l}{p^n}} \right\}\cup\{jp^n\mid j\geq 0\}\\
        = & ~\left\langle p^n, \left(\sum_{k\in I}\ceil*{\frac{in_k}{p^n}}d_k - \ceil*{\frac{in_l}{p^n}}\right)p^n + in_l : 1\leq i\leq p^n-1 \right\rangle.
    \end{align*}
\end{proof}

\begin{corollary}\label{m_and_F}
	Suppose that $l\in I$ and $d_l = 1$. Then 
	$$m_{H(Q_l)} = \begin{cases}
		p^n, \text{ if } |I|\geq 2,\\
		\min\{p^n,n_l\}, \text{ if } |I| = 1,
	\end{cases}$$
	and $F_{H(Q_l)} =  \left(m - \sum_{k\in I}\floor*{\frac{n_k}{p^n}}d_k + \floor*{\frac{n_l}{p^n}}-1\right)p^n - n_l$.
\end{corollary}
\begin{proof}
	If $|I|\geq 2$, then $\sum_{k\in I}\ceil*{\frac{in_k}{p^n}}d_k - \ceil*{\frac{in_l}{p^n}}\geq 1$ for all $1\leq i\leq p^n-1$. It follows from Corollary \ref{semigroup} that  $m_{H(Q_l)} = p^n$. 
	If $|I| = 1$, then $\sum_{k\in I}\ceil*{\frac{in_k}{p^n}}d_k - \ceil*{\frac{in_l}{p^n}} = 0$ for all $1\leq i\leq p^n-1$. By Corollary \ref{semigroup}, we have $m_{H(Q_l)} = \min\{p^n,n_0\}$.
    
    Note that $\sum_{k\in I}\ceil*{\frac{in_k}{p^n}}d_k - \ceil*{\frac{in_l}{p^n}}-1$ is increasing with respect to $i$.
	From the definition $F_{H(Q_l)}$ and \eqref{description2}, we get that 
	\begin{align*}
		F_{H(Q_l)} &= \left(\sum_{k\in I}\ceil*{\frac{(p^n-1)n_k}{p^n}}d_k - \ceil*{\frac{(p^n-1)n_l}{p^n}}-1\right)p^n  + (p^n-1)n_l\\
		& = \left(\sum_{k\in I}n_k d_k - \sum_{k\in I}\floor*{\frac{n_k}{p^n}}d_k - n_l +  \floor*{\frac{n_l}{p^n}}-1\right)p^n + (p^n-1)n_l\\
		& = \left(m - \sum_{k\in I}\floor*{\frac{n_k}{p^n}}d_k + \floor*{\frac{n_l}{p^n}}-1\right)p^n - n_l.
	\end{align*}
\end{proof}

For a totally ramified place $Q_l$ of degree one in $F/K(x)$, now we establish a necessary and sufficient condition for $H(Q_l)$ to be symmetric.
\begin{theorem}\label{sym_thm}
	Suppose that $l\in I$ and $d_l=1$. Then $H(Q_l)$ is symmetric if and only if $n_k \equiv -1 \pmod {p^n}$ for all $k\in I \setminus \{l\}$.
\end{theorem}
\begin{proof}
	It follows from Corollary \ref{m_and_F} that 
	$$F_{H(Q_l)} =  \left(m - \sum_{k\in I}\floor*{\frac{n_k}{p^n}}d_k + \floor*{\frac{n_l}{p^n}}-1\right)p^n - n_l.$$
	Since $2g-1\geq F_{H(Q_l)}$, we have 
	\begin{align*}
		&~2g-1 - F_{H(Q_l)}\\
		=&~(p^n-1)\left(m + \sum_{k\in I}d_k - 2\right) - 1 - \left(m - \sum_{k\in I}\floor*{\frac{n_k}{p^n}}d_k + \floor*{\frac{n_l}{p^n}}-1\right)p^n + n_l\\
		=&~mp^n - m + (p^n-1)\sum_{k\in I}d_k - 2p^n + 1 -mp^n + p^n \sum_{k\in I} \ceil*{\frac{n_k}{p^n}}d_k - p^n\ceil*{\frac{n_l}{p^n}} + p^n + n_l\\
		=&~\sum_{k\in I}\left(p^n\ceil*{\frac{n_k}{p^n}} + p^n-1-n_k\right)d_k - \left(p^n\ceil*{\frac{n_l}{p^n}} +  p^n-1-  n_l\right)\\
		=&~\sum_{k\in I,k\neq l}\left(p^n\ceil*{\frac{n_k}{p^n}} + p^n-1-n_k\right)d_k \geq 0.
	\end{align*}
	Note that the equality holds if and only if 
	$$p^n\ceil*{\frac{n_k}{p^n}} + p^n-1-n_k = 0 \text{ for all } k\in I\setminus\{l\},$$
	which is equivalent to  $n_k \equiv -1 \pmod {p^n}$ for all $k\in I \setminus \{l\}$.
	Thus $H(Q_l)$ is symmetric if and only if $n_k \equiv -1 \pmod {p^n}$ for all $k\in I \setminus \{l\}$.
\end{proof}

\section{The minimal generating sets of Weierstrass semigroups on linearized function fields}\label{section4}
In this section, we restrict our attention to the minimal generating sets of Weierstrass semigroups at several totally ramified places of degree one in $F/K(x)$ defined by the equation \eqref{equation1}.
Suppose that $v:=\min\{|I|, |K|\}$ and that $\{l_1,\dots,l_{|I|}\}$ is a permutation of the set $I$. We present our main results as follows.
\begin{theorem}\label{thm_minimal_set}
	Let $\lambda\in\mathbb{Z}$ with $\gcd(p,\lambda)=1$. Suppose that $2\leq t\leq v$ and $d_{l_k} = 1$ for all $1\leq k\leq t$. Then
	\begin{align*}
		\Gamma(Q_{l_1},\dots, Q_{l_t}) = \Bigg\{
		(j_1p^n-i\lambda n_{l_1},\dots, j_tp^n-i\lambda n_{l_t})\in \mathbb{N}^{t} ~\Big| ~1\leq i \leq p^n-1 , \\
		 j_k\geq \ceil*{\frac{i\lambda n_{l_k}}{p^n}} \text{ for } 1\leq k\leq t,~ \sum_{k=1}^{t}j_k= m\ceil*{\frac{i\lambda}{p^n}}-\sum_{k\in I}\floor*{\frac{i\lambda n_k}{p^n}}d_k + \sum_{k=1}^{t}\floor*{\frac{i\lambda n_{l_k}}{p^n}}\Bigg\}.
		\end{align*}
	Moreover, for each $(j_1p^n-i\lambda n_{l_1},\dots, j_tp^n - i\lambda n_{l_t})\in \Gamma(Q_{l_1},\dots,Q_{l_t})$, let 
	$$z = y^{-i\lambda}\left(\prod_{k\in J}q_k(x)^{m_k\ceil*{\frac{i\lambda}{p^n}}}\right)\left(\prod_{k=2}^{p^n}\frac{y}{y-\beta_k}\right)^{\ceil*{\frac{i\lambda}{p^n}}},$$
	then
	\begin{align*}
		\left(\left(\prod_{k=1}^{t}p_{l_k}(x)^{-j_k}\right)\left(\prod_{k=t+1}^{|I|}p_{l_k}(x)^{-\floor*{\frac{i\lambda n_{l_k}}{p^n}}}\right)z\right)_\infty = \sum_{k=1}^{t} (j_kp^n-i\lambda n_{l_k}) Q_{l_k}.
	\end{align*}
	In addition, we have $\Gamma(Q_{l_1},\dots,Q_{l_t}) \neq \varnothing$.
\end{theorem}

By choosing different $\lambda$, we can derive different expressions for $\Gamma(Q_{l_1},\dots,Q_{l_t})$.
In particular, taking $\lambda = 1$, we immediately obtain the following corollary, which is similar to the expression \eqref{description2}.
\begin{corollary}\label{Gamma_simply}
	Suppose that $2\leq t\leq v$ and $d_{l_k} = 1$ for all $1\leq k\leq t$. Then 
	\begin{align*}
		\Gamma(Q_{l_1},\dots, Q_{l_t}) = \Bigg\{
		(j_1p^n + in_{l_1},\dots, j_tp^n + in_{l_t})\in \mathbb{N}^{t} ~\Big| ~1\leq i \leq p^n-1 , \\
		 j_k\geq \ceil*{\frac{-in_{l_k}}{p^n}} \text{ for } 1\leq k\leq t,~ \sum_{k=1}^{t}j_k= \sum_{k\in I}\ceil*{\frac{i n_k}{p^n}}d_k - \sum_{k=1}^{t}\ceil*{\frac{i n_{l_k}}{p^n}}\Bigg\}.
		\end{align*}
	Moreover, for each $(j_1p^n + in_{l_1},\dots, j_tp^n + in_{l_t})\in \Gamma(Q_{l_1},\dots,Q_{l_t})$,
	\begin{align*}
		\left(\left(\prod_{k=1}^{t}p_{l_k}(x)^{-j_k}\right)\left(\prod_{k=t+1}^{|I|}p_{l_k}(x)^{\ceil*{\frac{i n_{l_k}}{p^n}}}\right)y^i\right)_\infty = \sum_{k=1}^{t} (j_kp^n+i n_{l_k}) Q_{l_k}.
	\end{align*}
\end{corollary}

When some of $n_{l_k}$ are equal modulo $p^n$, we can get an expression for $\Gamma(Q_{l_1},\dots,Q_{l_t})$, which is similar to the expression \eqref{description3}.
\begin{corollary}\label{Gamma_some_equal}
	Suppose that $2\leq t\leq v$ and $d_{l_k} = 1$ for all $1\leq k\leq t$. Suppose that $1\leq u\leq t$ and $n_{l_1} \equiv \dots \equiv n_{l_u} \pmod {p^n}$. 
	Let $\lambda\in\mathbb{Z}$ be the inverse of $-n_{l_1}$ modulo $p^n$. Then 
	
	\begin{align*}
		\Gamma(Q_{l_1},\dots, Q_{l_t}) = \Bigg\{
		(j_1p^n+i,\dots,j_up^n + i, j_{u+1}p^n - i\lambda n_{l_{u+1}}, \dots, j_tp^n-i\lambda n_{l_t})\in \mathbb{N}^{t} ~\Big|\\
		1\leq i \leq p^n-1,\, j_k\geq 0 \text{ for } 1\leq k\leq u,\, j_k\geq \ceil*{\frac{i\lambda n_{l_k}}{p^n}} \text{ for } u+1\leq k\leq t,\\
		 \sum_{k=1}^{t}j_k= m\ceil*{\frac{i\lambda}{p^n}}-\sum_{k\in I}\floor*{\frac{i\lambda n_k}{p^n}}d_k - u + \sum_{k=u+1}^{t}\floor*{\frac{i\lambda n_{l_k}}{p^n}}\Bigg\}.
		\end{align*}
	For each $(j_1p^n+i,\dots,j_up^n + i, j_{u+1}p^n - i\lambda n_{l_{u+1}}, \dots, j_tp^n-i\lambda n_{l_t})\in \Gamma(Q_{l_1},\dots,Q_{l_t})$, let 
	$$z = y^{-i\lambda}\left(\prod_{k\in J}q_k(x)^{m_k\ceil*{\frac{i\lambda}{p^n}}}\right)\left(\prod_{k=2}^{p^n}\frac{y}{y-\beta_k}\right)^{\ceil*{\frac{i\lambda}{p^n}}},$$
	then
	\begin{align*}
		&\left(\left(\prod_{k=1}^{t}p_{l_k}(x)^{-j_k}\right)\left(\prod_{k=t+1}^{|I|}p_{l_k}(x)^{-\floor*{\frac{i\lambda n_{l_k}}{p^n}}}\right)z\right)_\infty\\
	=~& \sum_{k=1}^{u} (j_kp^n+i) Q_{l_k} + \sum_{k=u+1}^{t} (j_kp^n-i\lambda n_{l_k}) Q_{l_k}.
	\end{align*}
\end{corollary}
\begin{proof}
	It follows from Theorem \ref{thm_minimal_set} that 
	\begin{align*}
		\Gamma(Q_{l_1},\dots, Q_{l_t}) = \Bigg\{
		(j'_1p^n-i\lambda n_{l_1},\dots, j'_tp^n-i\lambda n_{l_t})\in \mathbb{N}^{t} ~\Big| ~1\leq i \leq p^n-1 , \\
		 j'_k\geq \ceil*{\frac{i\lambda n_{l_k}}{p^n}} \text{ for } 1\leq k\leq t,~ \sum_{k=1}^{t}j'_k= m\ceil*{\frac{i\lambda}{p^n}}-\sum_{k\in I}\floor*{\frac{i\lambda n_k}{p^n}}d_k + \sum_{k=1}^{t}\floor*{\frac{i\lambda n_{l_k}}{p^n}}\Bigg\}.
		\end{align*}
		For each $1\leq k\leq u$, let $b_k\in \mathbb{Z}$ such that $\lambda n_{l_k} = -1 + b_kp^n$. Then for each $1\leq i\leq p^n-1$, we have $\ceil*{\frac{i\lambda n_{l_k}}{p^n}} = ib_k$ and $\floor*{\frac{i\lambda n_{l_k}}{p^n}} = ib_k - 1$.
		Let $j_k = j'_k - ib_k$ for $1\leq k\leq u$ and $j_k = j'_k$ for $u+1\leq k\leq t$. Thus we obtain the desired expression.
\end{proof}

We will prove Theorem \ref{thm_minimal_set} by induction. Firstly, we show that Theorem \ref{thm_minimal_set} holds for $t=2$.
\begin{proposition}\label{Gamma_two_places}
	Suppose that $d_{l_1} = d_{l_2} = 1$. Then for all $\lambda\in\mathbb{Z}$ with $\gcd(p,\lambda) = 1$,
	\begin{align*}
		\Gamma (Q_{l_1},Q_{l_2}) = &\Bigg\{  (j_1p^n - i\lambda n_{l_1}, j_2 p^n - i\lambda n_{l_2})\in \mathbb{N}^2 \mid 1\leq i\leq p^n-1,\, j_1\geq \ceil*{\frac{i\lambda n_{l_1}}{p^n}},\\
		& j_2\geq \ceil*{\frac{i\lambda n_{l_2}}{p^n}},\, j_1+j_2 = m\ceil*{\frac{i\lambda}{p^n}} - \sum_{k\in I}\floor*{\frac{i\lambda n_{l_k}}{p^n}}d_k + \floor*{\frac{i\lambda n_{l_1}}{p^n}} + \floor*{\frac{i\lambda n_{l_2}}{p^n}}\Bigg\}.
	\end{align*}
\end{proposition}
\begin{proof}
	Define the set
	\begin{align*}
		\Gamma = &\Bigg\{  (j_1p^n - i\lambda n_{l_1}, j_2 p^n - i\lambda n_{l_2})\in \mathbb{N}^2 \mid 1\leq i\leq p^n-1,\, j_1\geq \ceil*{\frac{i\lambda n_{l_1}}{p^n}},\\
		& j_2\geq \ceil*{\frac{i\lambda n_{l_2}}{p^n}},\, j_1+j_2 = m\ceil*{\frac{i\lambda}{p^n}} - \sum_{k\in I}\floor*{\frac{i\lambda n_{l_k}}{p^n}}d_k + \floor*{\frac{i\lambda n_{l_1}}{p^n}} + \floor*{\frac{i\lambda n_{l_2}}{p^n}}\Bigg\}.
	\end{align*}
	Let $(j_1p^n - i\lambda n_{l_1}, j_2p^n - i\lambda n_{l_2})\in \Gamma$ and $z = y^{-i\lambda}\left(\prod_{k\in J}q_k(x)^{m_k\ceil*{\frac{i\lambda}{p^n}}}\right)\left(\prod_{k=2}^{p^n}\frac{y}{y-\beta_k}\right)^{\ceil*{\frac{i\lambda}{p^n}}}$.
	By the divisors (\ref{divisor1}), (\ref{divisor2}) and (\ref{divisor3}), we have the following principal divisor
	{\allowdisplaybreaks
	\begin{align*}
		&\left(p_{l_1}(x)^{-j_1}p_{l_2}(x)^{-j_2}\left(\prod_{k=3}^{|I|}p_{l_k}(x)^{-\floor*{\frac{i\lambda n_{l_k}}{p^n}}}\right)z\right)_F\\
		=~& -j_1 p^n Q_{l_1} + j_1 D_0 - j_2 p^n Q_{l_2} + j_2 D_0  - \sum_{k=3}^{|I|} p^n \floor*{\frac{i\lambda n_{l_k}}{p^n}} Q_{l_k} + \sum_{k=3}^{|I|}\floor*{\frac{i\lambda n_{l_k}}{p^n}}d_{l_k} D_0\\
		 ~& -i\lambda \sum_{k\in J} m_k R_{k1} + i\lambda n_{l_1} Q_{l_1} + i\lambda n_{l_2} Q_{l_2} + i\lambda \sum_{k=3}^{|I|} n_{l_k} Q_{l_k}\\
		 ~& + \ceil*{\frac{i\lambda}{p^n}} \sum_{k\in J}m_k\sum_{t=1}^{p^n} R_{kt} - \ceil*{\frac{i\lambda}{p^n}}\sum_{k\in J}m_k e_k D_0\\
		 ~& + \ceil*{\frac{i\lambda}{p^n}}(p^n-1)\sum_{k\in J}m_kR_{k1} - \ceil*{\frac{i\lambda}{p^n}}\sum_{t=2}^{p^n}\sum_{k\in J} m_k R_{kt}\\
		=~& \sum_{k\in J}\left(m_kp^n\ceil*{\frac{i\lambda}{p^n}} - i\lambda m_k\right)R_{k1} + \sum_{k=3}^{|I|}\left(i\lambda n_{l_k} - p^n\floor*{\frac{i\lambda n_{l_k}}{p^n}}\right)Q_{l_k}\\
		 ~& -(j_1p^n - i\lambda n_{l_1}) Q_{l_1} - (j_2p^n - i\lambda n_{l_2}) Q_{l_2}.
	\end{align*}}
	Since $m_kp^n\ceil*{\frac{i\lambda}{p^n}} - i\lambda m_k\geq 0$ for all $k\in J$ and $i\lambda n_{l_k} - p^n\floor*{\frac{i\lambda n_{l_k}}{p^n}}\geq 0$ for all $3\leq k \leq |I|$,
	we get $(j_1p^n - i\lambda n_{l_1}, j_2p^n - i\lambda n_{l_2})\in H(Q_{l_1}, Q_{l_2})$. This implies that $\Gamma\subseteq H(Q_{l_1},Q_{l_2})$.

		On the other hand, since $j_1\geq \ceil*{\frac{i\lambda n_{l_1}}{p^n}}$, $j_2\geq \ceil*{\frac{i\lambda n_{l_2}}{p^n}}$, and 
		$$j_1 + j_2 =  m\ceil*{\frac{i\lambda}{p^n}} - \sum_{k\in I}\floor*{\frac{i\lambda n_{l_k}}{p^n}}d_k + \floor*{\frac{i\lambda n_{l_1}}{p^n}} + \floor*{\frac{i\lambda n_{l_2}}{p^n}},$$ we have 
	\begin{align*}
		\ceil*{\frac{i\lambda n_{l_1}}{p^n}} \leq j_1 & \leq m\ceil*{\frac{i\lambda}{p^n}} - \sum_{k\in I}\floor*{\frac{i\lambda n_{l_k}}{p^n}}d_k + \floor*{\frac{i\lambda n_{l_1}}{p^n}} + \floor*{\frac{i\lambda n_{l_2}}{p^n}} - \ceil*{\frac{i\lambda n_{l_2}}{p^n}}\\
		& = m\ceil*{\frac{i\lambda}{p^n}} - \sum_{k\in I}\floor*{\frac{i\lambda n_{l_k}}{p^n}}d_k + \floor*{\frac{i\lambda n_{l_1}}{p^n}} - 1.
	\end{align*}
	Similarly, we also have $\ceil*{\frac{i\lambda n_{l_2}}{p^n}}\leq j_2 \leq m\ceil*{\frac{i\lambda}{p^n}} - \sum_{k\in I}\floor*{\frac{i\lambda n_{l_k}}{p^n}}d_k + \floor*{\frac{i\lambda n_{l_2}}{p^n}} - 1$. This conclude that $\Gamma \subseteq G(Q_{l_1})\times G(Q_{l_2})$. Therefore $\Gamma\subseteq (G(Q_{l_1})\times G(Q_{l_2}))\cap H(Q_{l_1},Q_{l_2})$.

	Moreover, the set $\Gamma$ can be seen as the graph of the bijective map $\theta : G(Q_{l_1})\rightarrow G(Q_{l_2})$ given by $\theta(j_1p^n - i\lambda n_{l_1}) = j_2 p^n - i\lambda n_{l_2}$, which defines a permutation $\tau$ of the set $\{1,2,\dots,g\}$.
	It follows from Lemma \ref{minimal_generating_two} that $\Gamma = \Gamma(Q_{l_1},Q_{l_2})$.
\end{proof}

To prove that Theorem \ref{thm_minimal_set} holds for $t\geq 3$, we need the following technical lemma.
\begin{lemma}\label{i_is_same}
	Suppose that $2\leq t\leq v$ and $d_{l_k} = 1$ for all $1\leq k\leq t$. Let $\lambda\in\mathbb{Z}$ with $\gcd (p,\lambda) = 1$, and let $\mathbf{a}= (a_1,\dots,a_t)\in \Gamma(Q_{l_1},\dots,Q_{l_t})$.
	Suppose that $$\mathbf{a} = (j_1p^n- i_1\lambda n_{l_1},\dots,j_tp^n - i_t\lambda n_{l_t}),$$
	where $ 1 \leq i_k \leq p^n-1 $ and $ j_k \geq \ceil*{\frac{i_k\lambda n_{l_k}}{p^n}} $ for $1 \leq k \leq t$.
	Then $i_1 = i_k$ for all $2\leq k\leq t$.
\end{lemma}
\begin{proof}
	Suppose that $i_1\neq i_k$ for some $1\leq k\leq t$. Without loss of generality, we assume that $k = 2$. By Lemma \ref{discrepancy}, we have that 
	\begin{equation}\label{dim1}
		\ell \left(\sum_{k=1}^{t}a_k Q_{l_k}\right) = \ell\left(\sum_{k=1}^{t} a_k Q_{l_k} - Q_{l_1}\right) + 1,
	\end{equation}
	\begin{equation}\label{dim2}
		\ell \left(\sum_{k=1}^{t}a_k Q_{l_k}\right) = \ell\left(\sum_{k=1}^{t} a_k Q_{l_k} - Q_{l_2}\right) + 1,
	\end{equation}
	\begin{equation}\label{dim3}
		\ell \left(\sum_{k=1}^{t}a_k Q_{l_k}\right) = \ell\left(\sum_{k=1}^{t} a_k Q_{l_k} -Q_{l_1}- Q_{l_2}\right) + 1.
	\end{equation}
	
	Let $0\leq u_1\leq p^n-1$. Note that $j_1p^n - i_1\lambda n_{l_1} - u_1n_{l_1} \equiv 0 \pmod {p^n}$ if and only if $u_1 \equiv -i_1\lambda \pmod {p^n}$.
	We have $\floor*{\frac{a_1 - u_1n_{l_1}}{p^n}} = \floor*{\frac{a_1 -1 - u_1n_{l_1}}{p^n}} +1$ if and only if $u_1 \equiv -i_1\lambda \pmod {p^n}$.
	Then by \eqref{dim1} and Corollary \ref{D_dim}, we obtain that
	\begin{equation}\label{dim4}
		\sum_{k=1}^{t}\floor*{\frac{a_k-u_1n_{l_k}}{p^n}}+\sum_{k=t+1}^{|I|}\floor*{\frac{-u_1n_{l_k}}{p^n}}d_{l_k}\geq 0.
	\end{equation}
	
	Let $0\leq u_2\leq p^n-1$. Note that $j_2p^n - i_2\lambda n_{l_2} - u_2 n_{l_2}\equiv 0\pmod {p^n}$ if and only if $u_2\equiv -i_2\lambda \pmod {p^n}$.
	We have $\floor*{\frac{a_2-u_2 n_{l_2}}{p^n}} = \floor*{\frac{a_2-1-u_2 n_{l_2}}{p^n}}+1$ if and only if $u_2\equiv -i_2\lambda \pmod {p^n}$.
	And we have $u_1\neq u_2$ since $i_1\neq i_2$. 
	Then by \eqref{dim3}, \eqref{dim4}, and Corollary \ref{D_dim}, we obtain that
	$$\sum_{k=1}^{t}\floor*{\frac{a_k-u_2n_{l_k}}{p^n}}+\sum_{k=t+1}^{|I|}\floor*{\frac{-u_2n_{l_k}}{p^n}}d_{l_k}< 0.$$
	On the other hand, by (\ref{dim2}) and  Corollary \ref{D_dim}, we get that
	$$\sum_{k=1}^{t}\floor*{\frac{a_k-u_2n_{l_k}}{p^n}}+\sum_{k=t+1}^{|I|}\floor*{\frac{-u_2n_{l_k}}{p^n}}d_{l_k}\geq 0,$$
	which is a contradiction. Thus $i_1 = i_k$ for all $2\leq k\leq t$.
\end{proof}

\begin{definition}
	Let $\lambda\in\mathbb{Z}$ with $\gcd(p,\lambda)=1$. Suppose that $2\leq t\leq v$ and $d_{l_k} = 1$ for all $1\leq k\leq t$.
	Define the set 
	\begin{align*}
		\Gamma_\lambda(Q_{l_1},\dots, Q_{l_t}) = \Bigg\{
		\mathbf{a}_{\mathbf{j},i\lambda} = (j_1p^n-i\lambda n_{l_1},\dots, j_tp^n-i\lambda n_{l_t})\in \mathbb{N}^{t} ~\Big| ~1\leq i \leq p^n-1 , \\
		 j_k\geq \ceil*{\frac{i\lambda n_{l_k}}{p^n}} \text{ for } 1\leq k\leq t,~ \sum_{k=1}^{t}j_k= m\ceil*{\frac{i\lambda}{p^n}}-\sum_{k\in I}\floor*{\frac{i\lambda n_k}{p^n}}d_k + \sum_{k=1}^{t}\floor*{\frac{i\lambda n_{l_k}}{p^n}}\Bigg\}.
		\end{align*}
\end{definition}
Our goal is to show that $\Gamma_\lambda(Q_{l_1},\dots, Q_{l_t}) = \Gamma(Q_{l_1},\dots, Q_{l_t})$ for all $\lambda\in\mathbb{Z}$ with $\gcd(p,\lambda) = 1$.
We proceed in two steps: first, we prove that $\Gamma_\lambda(Q_{l_1},\dots, Q_{l_t}) \subseteq \Gamma(Q_{l_1},\dots, Q_{l_t})$, and second, we show that $\Gamma(Q_{l_1},\dots, Q_{l_t})\subseteq \Gamma_\lambda(Q_{l_1},\dots, Q_{l_t})$.
\begin{proposition}\label{S_subseteq_Gamma}
	Let $\lambda\in\mathbb{Z}$ with $\gcd(p,\lambda)=1$. Suppose that $2\leq t\leq v$ and $d_{l_k} = 1$ for all $1\leq k\leq t$. Then
	$$\Gamma_\lambda(Q_{l_1},\dots,Q_{l_t})\subseteq \Gamma (Q_{l_1},\dots,Q_{l_t}).$$
	Moreover, for each $\mathbf{a}_{\mathbf{j},i\lambda} = (j_1p^n-i\lambda n_{l_1},\dots, j_tp^n - i\lambda n_{l_t})\in \Gamma_\lambda(Q_{l_1},\dots,Q_{l_t})$, let 
	$$z = y^{-i\lambda}\left(\prod_{k\in J}q_k(x)^{m_k\ceil*{\frac{i\lambda}{p^n}}}\right)\left(\prod_{k=2}^{p^n}\frac{y}{y-\beta_k}\right)^{\ceil*{\frac{i\lambda}{p^n}}},$$
	then
	\begin{align*}
		\left(\left(\prod_{k=1}^{t}p_{l_k}(x)^{-j_k}\right)\left(\prod_{k=t+1}^{|I|}p_{l_k}(x)^{-\floor*{\frac{i\lambda n_{l_k}}{p^n}}}\right)z\right)_\infty = \sum_{k=1}^{t} (j_kp^n-i\lambda n_{l_k}) Q_{l_k}.
	\end{align*}
\end{proposition}
\begin{proof}
	We will prove the proposition by induction on $t$.
	It follows from Proposition \ref{Gamma_two_places} that the statement is right for $t=2$.
	Assume that $\Gamma_\lambda(Q_{l_1},\dots,Q_{l_k})\subseteq \Gamma (Q_{l_1},\dots,Q_{l_k})$ holds for all $2\leq k\leq t-1$, where $t\geq 3$.
	Let $\mathbf{a}_{\mathbf{j},i\lambda} =  (j_1p^n-i\lambda n_{l_1},\dots, j_tp^n - i\lambda n_{l_t})\in \Gamma_\lambda(Q_{l_1},\dots,Q_{l_t})$ and $z = y^{-i\lambda}\left(\prod_{k\in J}q_k(x)^{m_k\ceil*{\frac{i\lambda}{p^n}}}\right)\left(\prod_{k=2}^{p^n}\frac{y}{y-\beta_k}\right)^{\ceil*{\frac{i\lambda}{p^n}}}$.
	By divisors (\ref{divisor1}), (\ref{divisor2}) and (\ref{divisor3}), we have 
	{\allowdisplaybreaks\begin{align*}
		~&\left(\left(\prod_{k=1}^{t}p_{l_k}(x)^{-j_k}\right)\left(\prod_{k=t+1}^{|I|}p_{l_k}(x)^{-\floor*{\frac{i\lambda n_{l_k}}{p^n}}}\right)z\right)_\infty\\
		=~& -\sum_{k=1}^{t}j_k p^n Q_{l_k} + \sum_{k=1}^{t}j_k D_0 - \sum_{k=t+1}^{|I|} p^n \floor*{\frac{i\lambda n_{l_k}}{p^n}} Q_{l_k} + \sum_{k=t+1}^{|I|}\floor*{\frac{i\lambda n_{l_k}}{p^n}}d_{l_k} D_0\\
		 ~& -i\lambda \sum_{k\in J} m_k R_{k1} + \sum_{k=1}^{t}i\lambda n_{l_k} Q_{l_k} + i\lambda \sum_{k=t+1}^{|I|} n_{l_k} Q_{l_k}\\
		 ~& + \ceil*{\frac{i\lambda}{p^n}}\sum_{k\in J}m_k\sum_{t=1}^{p^n} R_{kt} - \ceil*{\frac{i\lambda}{p^n}}\sum_{k\in J}m_k e_k D_0\\
		 ~& + \ceil*{\frac{i\lambda}{p^n}}(p^n-1)\sum_{k\in J}m_kR_{k1} - \ceil*{\frac{i\lambda}{p^n}}\sum_{t=2}^{p^n}\sum_{k\in J}m_k R_{kt}\\
		=~& \sum_{k\in J}\left(m_kp^n\ceil*{\frac{i\lambda}{p^n}} - i\lambda m_k\right)R_{k1} + \sum_{k=t+1}^{|I|}\left(i\lambda n_{l_k} - p^n\floor*{\frac{i\lambda n_{l_k}}{p^n}}\right)Q_{l_k}\\
		 ~& -\sum_{k=1}^{t}(j_kp^n - i\lambda n_{l_k}) Q_{l_k}.
	\end{align*}}
	Since $m_kp^n\ceil*{\frac{i\lambda}{p^n}} - i\lambda m_k\geq 0$ for all $k\in J$ and $i\lambda n_{l_k} - p^n\floor*{\frac{i\lambda n_{l_k}}{p^n}}\geq 0$ for all $t+1\leq k \leq |I|$,
	we have that $\mathbf{a}_{\mathbf{j},i\lambda} =  (j_1p^n-i\lambda n_{l_1},\dots, j_tp^n - i\lambda n_{l_t})\in H(Q_{l_1},\dots, Q_{l_t})$.

	In order to show that $\mathbf{a}_{\mathbf{j},i\lambda}\in \Gamma(Q_{l_1},\dots,Q_{l_t})$, it suffices to prove that $\mathbf{a}_{\mathbf{j},i\lambda}$ is minimal in $\{\mathbf{u} = (u_1,\dots,u_t)\in H(Q_{l_1},\dots,Q_{l_t})\mid u_1 = j_1p^n - i\lambda n_{l_1}\}$.
	Suppose that $\mathbf{a}_{\mathbf{j},i\lambda}$ is not minimal in $\{\mathbf{u} = (u_1,\dots,u_t)\in H(Q_{l_1},\dots,Q_{l_t})\mid u_1 = j_1p^n - i\lambda n_{l_1}\}$. Then there exists $\mathbf{a} = (a_1,\dots,a_t)\in H(Q_{l_1},\dots,Q_{l_t})$ with $a_1 = j_1p^n - i\lambda n_{l_1}$ and  $\mathbf{a}\prec \mathbf{a}_{\mathbf{j},i\lambda}$.
	Let $h\in F$ be such that $(h)_\infty = \sum_{k=1}^{t} a_k Q_{l_k}$. Note that $\mathbf{a}\prec\mathbf{a}_{\mathbf{j},i\lambda}$ gives $a_k<j_k - i\lambda n_{l_k}$ for some $2\leq k\leq t$. Without loss of generality, assume that $a_2<j_2p^n - i\lambda n_{l_2}$.
	
	We take $b_{t-1} = j_tp^n - p^n \floor*{\frac{i\lambda n_{l_t}}{p^n}} + a_{t-1}$ and $b_k = a_k$ for $1\leq k\leq t-2$. 
	Since $-a_t + \left(j_t- \floor*{\frac{i\lambda n_{l_t}}{p^n}}\right)p^n \geq  i\lambda n_{l_t} - p^n\floor*{\frac{i\lambda n_{l_t}}{p^n}}>0$, we have 
	$$\left(h\cdot p_{l_t}(x)^{j_t- \floor*{\frac{i\lambda n_{l_t}}{p^n}}} p_{l_{t-1}}(x)^{-j_t + \floor*{\frac{i\lambda n_{l_t}}{p^n}}}\right)_\infty = \sum_{k=1}^{t-1}b_k Q_{l_k}.$$
	Hence $\mathbf{b} = (b_1,\dots,b_{t-1}) \in H(Q_{l_1},\dots, Q_{l_{t-1}})$.

	Let $c_{t-1} = \left(j_{t-1} + j_t - \floor*{\frac{i\lambda n_{l_t}}{p^n}}\right)p^n - i\lambda n_{l_{t-1}}$ and  
	$$\mathbf{c} = (j_1p^n - i\lambda n_{l_1},\dots, j_{t-2}p^n - i\lambda n_{l_{t-2}},c_{t-1}).$$
	Then we obtain that $\mathbf{c}\in \Gamma_{\lambda} (Q_{l_1},\dots,Q_{l_{t-1}})$. It follows from the induction hypothesis that $\Gamma_\lambda (Q_{l_1},\dots,Q_{l_{t-1}})\subseteq \Gamma (Q_{l_1},\dots,Q_{l_{t-1}})$.
	Thus we get that $\mathbf{c}\in \Gamma (Q_{l_1},\dots,Q_{l_{t-1}})$. Then by Proposition \ref{minimal_for_all}, we conclude that $\mathbf{c}$ is minimal in the set $\{\mathbf{u} = (u_1,\dots,u_{t-1})\in H(Q_{l_1},\dots,Q_{l_{t-1}})\mid u_1 = j_1p^n - i\lambda n_{l_1}\}$.
	Now we have 
	\begin{align*}
	\mathbf{b} \in \{\mathbf{u} = (u_1,\dots,u_{t-1})\in H(Q_{l_1},\dots,Q_{l_{t-1}})\mid u_1 = j_1p^n - i\lambda n_{l_1}\} \text{ and } \mathbf{b} \prec \mathbf{c},
	\end{align*}
	which is a contradiction to the minimality of $\mathbf{c}$. It follows that $\mathbf{a}_{\mathbf{j},i\lambda} $ is minimal in $\{\mathbf{u} = (u_1,\dots,u_t)\in H(Q_{l_1},\dots,Q_{l_t})~|~u_1=j_1p^n - i\lambda n_{l_1}\}$.
	Therefore we have $\mathbf{a}_{\mathbf{j},i\lambda} = (j_1p^n-i\lambda n_{l_1},\dots, j_tp^n-i\lambda n_{l_t})\in \Gamma(Q_{l_1},\dots,Q_{l_t})$ and
	\begin{align*}
		\left(\left(\prod_{k=1}^{t}p_{l_k}(x)^{-j_k}\right)\left(\prod_{k=t+1}^{|I|}p_{l_k}(x)^{-\floor*{\frac{i\lambda n_{l_k}}{p^n}}}\right)z\right)_\infty = \sum_{k=1}^{t} (j_kp^n-i\lambda n_{l_k}) Q_{l_k}.
	\end{align*}
	\end{proof}

\begin{proposition}\label{Gamma_subseteq_S}
	Let $\lambda\in\mathbb{Z}$ with $\gcd(p,\lambda)=1$. Suppose that $2\leq t\leq v$ and $d_{l_k} = 1$ for all $1\leq k\leq t$. Then
	$$\Gamma (Q_{l_1},\dots,Q_{l_t}) \subseteq \Gamma_\lambda(Q_{l_1},\dots,Q_{l_t}).$$
\end{proposition}
\begin{proof}
	We will prove the proposition by induction on $t$.
	It follows from Proposition \ref{Gamma_two_places} that the statement is right for $t=2$.
	Assume that $\Gamma (Q_{l_1},\dots,Q_{l_k})\subseteq\Gamma_\lambda(Q_{l_1},\dots,Q_{l_k})$ holds for all $2\leq k\leq t-1$, where $t\geq 3$.
	Let $\mathbf{a} = (a_1,\dots,a_t)\in \Gamma (Q_{l_1},\dots,Q_{l_t})$, then there exists $h\in F$ such that $(h)_\infty = \sum_{k=1}^{t}a_k Q_{l_k}$.
	By Lemmas \ref{Gamma_cross} and \ref{i_is_same}, we write 
	$$\mathbf{a} = (a_1,\dots,a_t) = (j_1p^n - i\lambda n_{l_1},\dots, j_tp^n - i\lambda n_{l_t}),$$
	where $1\leq i\leq p^n-1$ and $\ceil*{\frac{i\lambda n_{l_k}}{p^n}}\leq j_k\leq m\ceil*{\frac{i\lambda}{p^n}} - \sum_{j\in I}\floor*{\frac{i\lambda n_j}{p^n}}d_j+\floor*{\frac{i\lambda n_{l_k}}{p^n}}- 1$.

	Without loss of generality, we assume that $j_2 - \floor*{\frac{i\lambda n_{l_2}}{p^n}} = \max_{2\leq k\leq t}\left\{j_k - \floor*{\frac{i\lambda n_{l_k}}{p^n}}\right\}$. We have
	$$\left(h\cdot(p_{l_t}(x)/p_{l_2}(x))^{j_t-\floor*{\frac{i\lambda n_{l_t}}{p^n}}}\right)_{\infty}= a_1Q_{l_1} + \left(a_2+p^n\left(j_t-\floor*{\frac{i\lambda n_{l_t}}{p^n}}\right)\right)Q_{l_2} +\sum_{k=3}^{t-1}a_k Q_{l_k}.$$
	Thus we have that $\left(a_1, a_2 + p^n\left(j_t - \floor*{\frac{i\lambda n_{l_t}}{p^n}}\right),a_3,\dots,a_{t-1}\right)\in H(Q_{l_1},\dots,Q_{l_{t-1}})$.
	By Theorem \ref{lub_semi_group}, there exists $\mathbf{b} = (b_1,\dots,b_{t-1})\in \tilde{\Gamma}(Q_{l_1},\dots,Q_{l_{t-1}})$ such that 
	$$\mathbf{b}\preceq \left(a_1, a_2 + p^n\left(j_t - \floor*{\frac{i\lambda n_{l_t}}{p^n}}\right),a_3,\dots,a_{t-1}\right),$$
	and $b_1 = a_1 = j_1p^n - i\lambda n_{l_1}$. If $b_2\leq a_2$, then $(b_1,b_2,\dots,b_{t-1},0)\prec \mathbf{a}$. This yields a contradiction as $\mathbf{a}$ is minimal in $\{\mathbf{u} = (u_1,\dots,u_t)\in H(Q_{l_1},\dots,Q_{l_t})\mid u_1 = j_1p^n - i\lambda n_{l_1}\}$.
	Thus we get $b_2> a_2 >0$. Let $M = \{k_1,\dots,k_u\} = \{1\leq k\leq t-1\mid b_k>0\}$. We have $u\geq 2$ since $b_1 = a_1>0$ and $b_2>a_2>0$. 
	Then $\pi_M(\mathbf{b})\in \Gamma(Q_{l_{k_1}},\dots,Q_{l_{k_u}})$. By the induction hypothesis, we obtain that
	$$\pi_M(\mathbf{b}) = (T_{k_1}p^n - i\lambda n_{l_{k_1}},\dots, T_{k_u}p^n - i\lambda n_{l_{k_u}})\in \Gamma_\lambda (Q_{l_{k_1}},\dots,Q_{l_{k_u}}),$$
	where $1\leq i\leq p^n-1$, $T_{k_j}\geq \ceil*{\frac{i\lambda n_{l_{k_j}}}{p^n}}$ for $1\leq j\leq u$, and $\sum_{j=1}^{u} T_{k_{j}} = m \ceil*{\frac{i\lambda}{p^n}} - \sum_{j\in I} \floor*{\frac{i\lambda n_j}{p^n}}d_j + \sum_{j=1}^{u} \floor*{\frac{i\lambda n_{l_{k_j}}}{p^n}}$.
	
	Note that $k_1 = 1$ as $b_1 = a_1 >0$ and $k_2 = 2$ as $b_2>a_2>0$. Since $mT_2 - i\lambda n_{l_2} = b_2 > a_2 = mj_2 - i\lambda n_{l_2}$, we have $T_2\geq j_2 + 1$.
	Since $j_2 - \floor*{\frac{i\lambda n_{l_2}}{p^n}}\geq j_t - \floor*{\frac{i\lambda n_{l_t}}{p^n}}$, we have that 
	$$T_2 - j_t + \floor*{\frac{i\lambda n_{l_t}}{p^n}} \geq j_2 + 1 - j_t + \floor*{\frac{i\lambda n_{l_t}}{p^n}} \geq 1 + \floor*{\frac{i\lambda n_{l_2}}{p^n}} = \ceil*{\frac{i\lambda n_{l_2}}{p^n}}.$$
	Set 
	$$\mathbf{c}  = \left(b_1, b_2 -\left(j_t - \floor*{\frac{i\lambda n_{l_t}}{p^n}}\right)p^n, b_3,\dots, b_{t-1}, j_tp^n - i\lambda n_{l_t} \right).$$
	We obtain that $\mathbf{c}\preceq \mathbf{a}$. Then $\pi_{M\cup\{t\}}(\mathbf{c})$ is formed by some of nonzero coordinates of $\mathbf{c}$. We verify that 
	$$T_1 + T_2 - j_t + \floor*{\frac{i\lambda n_{l_t}}{p^n}} + \sum_{j=3}^{u}T_{k_j} + j_t = m\ceil*{\frac{i\lambda}{p^n}} - \sum_{j\in I}\floor*{\frac{i\lambda n_j}{p^n}}d_j + \sum_{j=1}^{u}\floor*{\frac{i\lambda n_{l_{k_j}}}{p^n}} + \floor*{\frac{i\lambda n_{l_t}}{p^n}}.$$
	It follows that $\pi_{M\cup\{t\}}(\mathbf{c})\in \Gamma_\lambda (Q_{l_{k_1}},\dots,Q_{l_{k_u}},Q_{l_t})$. By Proposition \ref{S_subseteq_Gamma}, we have that $\Gamma_\lambda (Q_{l_{k_1}},\dots,Q_{l_{k_u}},Q_{l_t})\subseteq \Gamma(Q_{l_{k_1}},\dots,Q_{l_{k_u}},Q_{l_t})$.
	Then we get $\mathbf{c}\in \tilde{\Gamma} (Q_{l_{1}},\dots,Q_{l_t})\subseteq H(Q_{l_1},\dots,Q_{l_t})$. Note that $\mathbf{c}\preceq \mathbf{a}$ and $\mathbf{a}\in  \Gamma(Q_{l_t},\dots,Q_{l_t})$.
	Therefore, we obtain that $\mathbf{a} = \mathbf{c}$ as otherwise $\mathbf{a}$ is not minimal in $\{\mathbf{u} = (u_1,\dots,u_t)\in H(Q_{l_1},\dots,Q_{l_t})\mid u_1= j_1p^n - i\lambda n_{l_1}\}$.
	Since $a_k>0$ for all $1\leq k\leq t$, we obtain that $\mathbf{a} = \mathbf{c} = \pi_{M\cup\{t\}}(\mathbf{c})\in \Gamma (Q_{l_1},\dots,Q_{l_t})$. The proof is completed.
\end{proof}

Now we finish the proof of Theorem \ref{thm_minimal_set}.
\begin{proof}
	Combining Propositions \ref{S_subseteq_Gamma} and \ref{Gamma_subseteq_S}, it remains to show that 
	$\Gamma_1(Q_{l_1},\dots,Q_{l_t})\neq \varnothing$. For each $1\leq k\leq t$, let $b_{k}\in\mathbb{Z}$ and $1\leq c_{k}\leq p^n-1$ such that $-n_{l_k}= p^nb_{k} + c_{k}$.  
	Then we have $\floor*{\frac{n_{l_k}}{p^n}} = -b_{k}-1$ and $\ceil*{\frac{n_{l_k}}{p^n}} = -b_{k}$. Let $j_1 = m - \sum_{k\in I}\floor*{\frac{n_k}{p^n}}d_k -t - b_{1}$ and $j_k = -b_{k}$ for $2\leq k\leq t$.
	It is easy to verify that $(j_1p^n - n_{l_1},\dots, j_tp^n - n_{l_t})\in \Gamma_1(Q_{l_1},\dots,Q_{l_t})$.
\end{proof}

\section{Some examples}\label{section5}
In this section, we investigate a special class of linearized function fields and present several examples as applications of our main results. Consider the linearized function field defined by the equation \eqref{equation1}, where the denominator $g(x)$ is a separable polynomial.
In other words, we study the function field $F = K(x,y)$ defined by 
$$L(y) = \alpha\cdot\frac{\prod_{j=0}^{r}q_j(x)^{m_j}}{\prod_{i=0}^{s}p_i(x)},$$
where $\alpha\in K$, $m_1,\dots,m_r\in \mathbb{N}$, $p_1(x),\dots,p_s(x)$, $q_1(x),\dots,q_r(x)\in K[x]$ are pairwise distinct monic irreducible polynomials, $p_0(x) = q_0(x) = 1$, and $m_0 = \sum_{i=1}^{s} \deg p_i(x) - \sum_{j=1}^r m_j\cdot\deg q_j(x)$.
Recall that $d_k = \deg p_k(x)$ for $1\leq k\leq s$ and $n_0 = -m_0$. 
Let $Q_i$ be the only zero of $p_i(x)$ in $\mathbf{P}_F$ for $1\leq i\leq s$. Let $Q_\infty$ be the only pole of $x$ in $\mathbf{P}_F$ if $n_0>0$.
We will carry out our investigation and discussion by distinguishing the two cases $n_0\leq 0$ and $n_0>0$.

\subsection{The case $n_0\leq 0$}
In this subsection, we suppose that $n_0\leq 0$. 
\begin{proposition}\label{form1}
	Suppose that $s\geq 2$. Let $1\leq l\leq s$ with $d_l = 1$.  Then 

	(i) $G(Q_l) = \left\{jp^n + i\mid 1\leq i\leq p^n-1,\, 0\leq j\leq \sum_{k=1}^{s}d_k - 2\right\}$.

	(ii) $H(Q_l) = \{a\in \mathbb{N}\mid a\geq (\sum_{k=1}^{s}d_k - 1)p^n\}\cup\{kp^n\mid k\in\mathbb{N}_0\}$.

	(iii) $m_{H(Q_l)} = p^n$ and $F_{H(Q_l)} = (\sum_{k=1}^{s}d_k-1)p^n -1$.

	Moreover, suppose that $2\leq t\leq s$ and $d_i = 1$ for all $1\leq i\leq t$. Then 
	\begin{align*}
		\Gamma(Q_1,\dots,Q_t) = \Big\{ (j_1p^n+i,\dots, j_tp^n+i)&\in \mathbb{N}^{t} \mid 1\leq i \leq p^n-1 , \\
		 &j_k\geq 0 \text{ for } 1\leq k\leq t,~ \sum_{k=1}^{t}j_k = \sum_{k=t+1}^{s}d_k \Big\}.
	\end{align*}
\end{proposition}
\begin{proof}
	By \eqref{description2} in Proposition \ref{gap_set}, we get that
	\begin{align*}
		G(Q_l) & =  \left\{jp^n + i ~\Big | ~ 1\leq i\leq p^n-1,\, \ceil*{\frac{-i}{p^n}}\leq j\leq \sum_{k=1}^s\ceil*{\frac{i}{p^n}}d_k - \ceil*{\frac{i}{p^n}}-1\right\}\\
		&= \left\{jp^n + i\mid 1\leq i\leq p^n-1,\, 0\leq j\leq \sum_{k=1}^{s}d_k - 2\right\}.
	\end{align*}
	By Corollary \ref{semigroup}, we have 
	\begin{align*}
		H(Q_l) &= \left\langle p^n, \left(\sum_{k=1}^s\ceil*{\frac{i}{p^n}}d_k - \ceil*{\frac{i}{p^n}}\right)p^n+i : 1\leq i\leq p^n-1\right\rangle\\
		&= \left\langle p^n, \left(\sum_{k=1}^{s}d_k - 1\right)p^n+i : 1\leq i\leq p^n-1\right\rangle\\
		&= \left\{a\in \mathbb{N} ~\Big | ~ a\geq \left(\sum_{k=1}^{s}d_k - 1\right)p^n\right\}\cup\left\{kp^n\mid k\in\mathbb{N}_0\right\}.
	\end{align*}
	Note that $s\geq 2$. Then by Corollary \ref{m_and_F}, we obtain that $m_{H(Q_l)} = p^n$ and 
	$$F_{H(Q_l)} = \left(\sum_{k=1}^{s}d_k - \sum_{k=1}^s\floor*{\frac{1}{p^n}}d_k + \floor*{\frac{1}{p^n}}-1\right)p^n - n_l = \left(\sum_{k=1}^{s}d_k-1\right)p^n -1.$$

	Moreover, suppose that $2\leq t\leq s$ and $d_i = 1$ for all $1\leq i\leq t$. The desired expression is obtained from Corollary \ref{Gamma_simply}. 
\end{proof}

	\begin{example}
		Consider the function field $F = \mathbb{F}_8(x,y)$ defined by 
		$$y^8 + y^4 + y^2 + y = \frac{x(x+1)}{x^3+x+1}.$$
		Then by Proposition \ref{form1}, for each $1\leq l\leq 3$, we obtain that 
		$$G(Q_l) = \{1,2,3,4,5,6,7,9,10,11,12,13,14,15\},$$
		$$H(Q_l) = \{a\in\mathbb{N}\mid a\geq 16\}\cup\{0,8\},$$
		$m_{H(Q_l)} = 8$, and $F_{H(Q_l)} = 15$. Moreover, we have 
		$$\Gamma(Q_1,Q_2) = \left\{\begin{array}{l}(1,9),(9,1),(2,10),(10,2),(3,11),(11,3),(4,12),(12,4)\\
		(5,13),(13,5),(6,14),(14,6),(7,15),(15,7)\end{array}\right\}, \text{ and }$$
		$$\Gamma(Q_1,Q_2,Q_3) = \left\{\begin{array}{l}(1,1,1),(2,2,2),(3,3,3),(4,4,4)\\
		(5,5,5),(6,6,6),(7,7,7)\end{array}\right\}.$$
	\end{example}

\subsection{The case $n_0>0$}
	In this subsection, we suppose that $n_0 > 0$. 

\begin{proposition}\label{form2}
	For the infinite place $Q_\infty$, 

	(i) $G(Q_\infty) = \Big\{jp^n + in_0 \mid 1\leq i\leq p^n-1,\, \ceil*{\frac{-in_0}{p^n}}\leq j\leq \sum_{i=1}^{s}d_k-1\Big\}$.

	(ii) $H(Q_\infty) =  \left\langle p^n , p^n\sum_{k=1}^{s}d_k + in_0 : 1\leq i\leq p^n-1\right\rangle$.

	(iii) $m_{H(Q_\infty)} = p^n$ and $F_{H(Q_\infty)} = p^n\left(\sum_{k=1}^{s}d_k + n_0-1\right) - n_0$.

	Moreover, suppose that $1\leq t\leq s$ and $d_i = 1$ for all $1\leq i\leq t$. Then 
	\begin{align*}
		\Gamma(Q_\infty, Q_1,\dots,Q_t) = \Bigg\{ (j_0 + in_0, j_1p^n+i,\dots, j_tp^n+i)\in \mathbb{N}^{t+1} \mid 1\leq i \leq p^n-1 , \\
		 j_0\geq \ceil*{\frac{-in_0}{p^n}},\, j_k\geq 0 \text{ for } 1\leq k\leq t,~ \sum_{k=0}^{t}j_k = \sum_{k=t+1}^{s}d_k  \Bigg\}.
	\end{align*}
\end{proposition}
\begin{proof}
	By \eqref{description2} in Proposition \ref{gap_set}, the set $G(Q_\infty)$ is given by
	\begin{align*}
		&\left\{jp^n + in_0 ~\Big | ~ 1\leq i\leq p^n-1,\, \ceil*{\frac{-in_0}{p^n}}\leq j\leq \sum_{k=1}^s\ceil*{\frac{i}{p^n}}d_k + \ceil*{\frac{in_0}{p^n}} - \ceil*{\frac{in_0}{p^n}}-1\right\}\\
		 = &\left\{jp^n + i\mid 1\leq i\leq p^n-1,\, \ceil*{\frac{-in_0}{p^n}}\leq j\leq \sum_{k=1}^{s}d_k-1\right\}.
	\end{align*}
	By Corollary \ref{semigroup}, we have 
	\begin{align*}
		H(Q_\infty) & = \left\langle p^n, \left(\sum_{k=1}^s \ceil*{\frac{i}{p^n}}d_k  + \ceil*{\frac{in_0}{p^n}} - \ceil*{\frac{in_0}{p^n}}\right)p^n + i : 1\leq i\leq p^n-1 \right\rangle\\
		&= \left\langle p^n, p^n\sum_{k=1}^{s}d_k +i : 1\leq i\leq p^n-1\right\rangle.
	\end{align*}
	Note that $n_0>0$ and $s\geq 1$. Then by Corollary \ref{m_and_F}, we have $m_{H(Q_l)} = p^n$ and 
	\begin{align*}
	 F_{H(Q_l)} &= \left(\sum_{k=1}^{s}d_k + n_0 - \sum_{k = 1}^s\floor*{\frac{1}{p^n}}d_k - \floor*{\frac{n_0}{p^n}} + \floor*{\frac{n_0}{p^n}}-1\right)p^n - n_0 \\
	 			&= \left(\sum_{k=1}^{s}d_k + n_0 -1\right)p^n-n_0.
	\end{align*}

	Moreover, suppose that $1\leq t\leq s$ and $d_i = 1$ for all $1\leq i\leq t$. Take $\lambda = -1$ in Corollary \ref{Gamma_some_equal}. We get the desired expression.
\end{proof}

In particular, let $s=0$. Then the linearized function field $F$ is defined by $L(y) = f(x)$, where $f(x)\in K[x]$ with $\gcd(\deg f(x), p) = 1$. This class of function fields contains many important families, including Hermitian function fields and function fields of norm-trace curves. We have the following result.
\begin{corollary}
	Suppose that $s = 0$. Then $H(Q_\infty) = \langle p^n, n_0\rangle$ and $H(Q_\infty)$ is symmetric.
\end{corollary}
\begin{proof}
	It follows from Proposition \ref{form3} that 
	 $$H(Q_\infty) =  \left\langle p^n , in_0 : 1\leq i\leq p^n-1\right\rangle = \langle p^n, n_0\rangle.$$
	By Theorem \ref{sym_thm}, we have that $H(Q_\infty)$ is symmetric.
\end{proof}

\begin{proposition}\label{form3}
	Suppose that $1\leq l\leq s$ and $d_l = 1$. Then 

	(i) $G(Q_l) = \left\{jp^n + i\mid 1\leq i\leq p^n-1,\, 0\leq j\leq \sum_{k=1}^{s}d_k + \ceil*{\frac{in_0}{p^n}} - 2\right\}$.

	(ii) $H(Q_l) = \left\langle p^n , \left(\sum_{k=1}^{s}d_k + \ceil*{\frac{in_0}{p^n}} - 1\right)p^n + i : 1\leq i\leq p^n-1\right\rangle$.

	(iii) $m_{H(Q_l)} = p^n$ and $F_{H(Q_l)} = \left(\sum_{k=1}^{s}d_k + n_0 - \floor*{\frac{n_0}{p^n}}-1\right)p^n-1$.

	Moreover, suppose that $2\leq t\leq s$ and $d_i = 1$ for all $1\leq i\leq t$. Then 
	\begin{align*}
		\Gamma(Q_1,\dots,Q_t) = \Bigg\{ (j_1p^n+i,\dots, j_tp^n+i)\in \mathbb{N}^{t} \mid 1\leq i \leq p^n-1 , \\
		 j_k\geq 0 \text{ for } 1\leq k\leq t,~ \sum_{k=1}^{t}j_k = \sum_{k=t+1}^{s}d_k + \ceil*{\frac{in_0}{p^n}} \Bigg\}.
	\end{align*}
\end{proposition}
\begin{proof}
	The proof is similar to Propositions \ref{form1} and \ref{form2}, thus it is omitted.
\end{proof}

	\begin{example}
		Consider the function field $F = \mathbb{F}_8(x,y)$ defined by 
		$$y^4 + y^2 + y = \frac{x^2(x+1)^2(x^2+x+1)}{x^3+x+1}.$$
		Then by Proposition \ref{form2}, we get that
		$$G(Q_\infty) = \{1,2,3,5,6,7,9,10,11,13,14,17\},$$
		$m_{H(Q_\infty)} = 4$ and $F_{H(Q_\infty)} = 17$. 
		By Proposition \ref{form3}, for each $1\leq l\leq 3$, we obtain that
		$$G(Q_l) = \{1,2,3,5,6,7,9,10,11,14,15,19\},$$
		$m_{H(Q_l)} = 4$ and $F_{H(Q_l)} = 19$. 
		Moreover, we have 
		$$\Gamma(Q_1,Q_2) = \left\{\begin{array}{l}(1,9),(9,1),(5,5),(2,14),(14,2),(6,10),(10,6),\\
			(3,19),(19,3),(7,15),(15,7),(11,11)\end{array}\right\},$$
		$$\Gamma(Q_1,Q_2,Q_3) = \left\{\begin{array}{l}(5,1,1),(1,5,1),(5,1,5),(10,2,2),(2,10,2),(2,2,10),\\
			(6,6,2),(6,2,6),(2,6,6),(15,3,3),(3,15,3),(3,3,15),(11,7,3),\\
			(11,3,7),(7,11,3),(7,3,11),(3,11,7),(3,7,11),(7,7,7)\end{array}\right\},$$
		and $\Gamma(Q_\infty,Q_1,Q_2,Q_3)$ is given by 
		$$\left\{\begin{array}{l}(3,1,1,1),(2,6,2,2),(2,2,6,2),(2,2,2,6),(6,2,2,2),\\
		(1,11,3,3),(1,3,11,3),(1,3,3,11),(1,7,7,3),(1,7,3,7),(1,3,7,7),\\
		(5,7,3,3),(5,3,7,3),(5,3,3,7),(9,3,3,3)\end{array}\right\}.$$
	\end{example}

	\section*{Acknowledgements}
	{\sloppy\par This work is supported by Guangdong Basic and Applied Basic Research Foundation (No. 2025A1515011764), the National Natural Science Foundation of China (No. 12441107), and the National Key Research and Development Program of China (No. 2025YFA1017100)\par}
    
	\bibliographystyle{elsarticle-num}
	\bibliography{refs}

@book{stichtenothAlgebraicFunctionFields2009,
  title = {Algebraic {{Function Fields}} and {{Codes}}},
  author = {Stichtenoth, Henning},
  year = {2009},
  series = {Graduate {{Texts}} in {{Mathematics}}},
  number = {254},
  publisher = {Springer Berlin Heidelberg},
  doi = {10.1007/978-3-540-76878-4},
  isbn = {978-3-540-76877-7 978-3-540-76878-4},
  edition = {Second},
  address = {Berlin}
}

@article{navarroBasesRiemann2024,
  title = {Bases for {{Riemann}}--{{Roch}} Spaces of Linearized Function Fields with Applications to Generalized Algebraic Geometry Codes},
  author = {Navarro, Horacio},
  year = 2024,
  journal = {Designs, Codes and Cryptography},
  volume = {92},
  number = {10},
  pages = {3033--3048},
  issn = {0925-1022, 1573-7586},
  doi = {10.1007/s10623-024-01426-6}
}

@article{castellanosWeierstrassSemigroups2024,
  title = {Weierstrass Semigroups, Pure Gaps and Codes on Function Fields},
  author = {Castellanos, Alonso S. and Mendoza, Erik A. R. and Quoos, Luciane},
  year = 2024,
  journal = {Designs, Codes and Cryptography},
  volume = {92},
  number = {5},
  pages = {1219--1242},
  issn = {0925-1022, 1573-7586},
  doi = {10.1007/s10623-023-01339-w}
}

@article{hommaWeierstrassSemigroup1996,
  title = {The {{Weierstrass}} Semigroup of a Pair of Points on a Curve},
  author = {Homma, Masaaki},
  year = 1996,
  journal = {Archiv der Mathematik},
  volume = {67},
  number = {4},
  pages = {337--348},
  issn = {0003-889X, 1420-8938},
  doi = {10.1007/BF01197599}
}

@incollection{matthewsWeierstrassSemigroup2004,
  title = {The {{Weierstrass Semigroup}} of an $m$-Tuple of {{Collinear Points}} on a {{Hermitian Curve}}},
  booktitle = {Finite {{Fields}} and {{Applications}}},
  author = {Matthews, Gretchen L.},
  editor = {Goos, Gerhard and Hartmanis, Juris and Van Leeuwen, Jan and Mullen, Gary L. and Poli, Alain and Stichtenoth, Henning},
  year = 2004,
  volume = {2948},
  pages = {12--24},
  publisher = {Springer Berlin Heidelberg},
  address = {Berlin, Heidelberg},
  doi = {10.1007/978-3-540-24633-6_2},
  isbn = {978-3-540-21324-6 978-3-540-24633-6}
}

@article{castellanosWeierstrassSemigroup2018,
  title = {On {{Weierstrass}} Semigroup at m Points on Curves of the Form $f(y)=g(x)$},
  author = {Castellanos, A.S. and Tizziotti, G.},
  year = 2018,
  journal = {Journal of Pure and Applied Algebra},
  volume = {222},
  number = {7},
  pages = {1803--1809},
  issn = {00224049},
  doi = {10.1016/j.jpaa.2017.08.007}
}

@article{kimIndexWeierstrass1994,
  title = {On the Index of the {{Weierstrass}} Semigroup of a Pair of Points on a Curve},
  author = {Kim, Seon Jeong},
  year = 1994,
  journal = {Archiv der Mathematik},
  volume = {62},
  number = {1},
  pages = {73--82},
  issn = {0003-889X, 1420-8938},
  doi = {10.1007/BF01200442}
}

@article{castellanosWeierstrassSemigroupsPure2024,
  title = {Weierstrass Semigroups, Pure Gaps and Codes on Function Fields},
  author = {Castellanos, Alonso S. and Mendoza, Erik A. R. and Quoos, Luciane},
  year = {2024},
  journal = {Designs, Codes and Cryptography},
  volume = {92},
  number = {5},
  pages = {1219--1242},
  issn = {0925-1022, 1573-7586},
  doi = {10.1007/s10623-023-01339-w}
}

@article{mendozaKummerExtensions2023,
  title = {On {{Kummer}} Extensions with One Place at Infinity},
  author = {Mendoza, Erik A.R.},
  year = 2023,
  journal = {Finite Fields and Their Applications},
  volume = {89},
  pages = {102209},
  issn = {10715797},
  doi = {10.1016/j.ffa.2023.102209}
}

@article{abdonWeierstrassPoints2019,
  title = {Weierstrass Points on {{Kummer}} Extensions},
  author = {Abd{\'o}n, Miriam and Borges, Herivelto and Quoos, Luciane},
  year = 2019,
  journal = {Advances in Geometry},
  volume = {19},
  number = {3},
  pages = {323--333},
  issn = {1615-7168, 1615-715X},
  doi = {10.1515/advgeom-2018-0021}
}

@article{beelenWeierstrassSemigroups2021,
  title = {Weierstrass Semigroups on the {{Skabelund}} Maximal Curve},
  author = {Beelen, Peter and Landi, Leonardo and Montanucci, Maria},
  year = 2021,
  journal = {Finite Fields and Their Applications},
  volume = {72},
  pages = {101811},
  issn = {10715797},
  doi = {10.1016/j.ffa.2021.101811}
}

@article{bartoliWeierstrassSemigroups2021a,
  title = {Weierstrass Semigroups at Every Point of the {{Suzuki}} Curve},
  author = {Bartoli, Daniele and Montanucci, Maria and Zini, Giovanni},
  year = 2021,
  journal = {Acta Arithmetica},
  volume = {197},
  number = {1},
  pages = {1--20},
  issn = {0065-1036, 1730-6264},
  doi = {10.4064/aa181203-24-2}
}

@article{beelenWeierstrassSemigroups2018,
  title = {Weierstrass Semigroups on the {{Giulietti}}--{{Korchm\'aros}} Curve},
  author = {Beelen, Peter and Montanucci, Maria},
  year = 2018,
  journal = {Finite Fields and Their Applications},
  volume = {52},
  pages = {10--29},
  issn = {10715797},
  doi = {10.1016/j.ffa.2018.03.002}
}

@article{guneriAutomorphismGroup2013,
  title = {The Automorphism Group of the Generalized {{Giulietti}}--{{Korchm\'aros}} Function Field},
  author = {G{\"u}neri, Cem and {\"O}zdemiry, Mehmet and Stichtenoth, Henning},
  year = 2013,
  journal = {Advances in Geometry},
  volume = {13},
  number = {2},
  pages = {369--380},
  issn = {1615-7168, 1615-715X},
  doi = {10.1515/advgeom-2012-0040}
}

@article{castellanosOneTwoPointCodes2016,
  title = {One- and {{Two-Point Codes Over Kummer Extensions}}},
  author = {Castellanos, Alonso S. and Masuda, Ariane M. and Quoos, Luciane},
  year = 2016,
  journal = {IEEE Transactions on Information Theory},
  volume = {62},
  number = {9},
  pages = {4867--4872},
  issn = {0018-9448, 1557-9654},
  doi = {10.1109/TIT.2016.2583437}
}

@article{montanucciAGCodes2020,
  title = {{{AG}} Codes from the Second Generalization of the {{GK}} Maximal Curve},
  author = {Montanucci, Maria and Pallozzi Lavorante, Vincenzo},
  year = 2020,
  journal = {Discrete Mathematics},
  volume = {343},
  number = {5},
  pages = {111810},
  issn = {0012365X},
  doi = {10.1016/j.disc.2020.111810}
}

@article{garciaConsecutiveWeierstrass1993,
  title = {Consecutive {{Weierstrass}} Gaps and Minimum Distance of {{Goppa}} Codes},
  author = {Garcia, Arnaldo and Kim, Seon Jeong and Lax, Robert F.},
  year = 1993,
  journal = {Journal of Pure and Applied Algebra},
  volume = {84},
  number = {2},
  pages = {199--207},
  issn = {00224049},
  doi = {10.1016/0022-4049(93)90039-V}
}

@article{niemannNonisomorphicMaximal2025,
  title = {Non-Isomorphic Maximal Function Fields of Genus $q - 1$},
  author = {Niemann, Jonathan},
  year = 2025,
  journal = {Finite Fields and Their Applications},
  volume = {106},
  pages = {102618},
  issn = {10715797},
  doi = {10.1016/j.ffa.2025.102618}
}

@article{beelenFamilyNonisomorphic2025,
  title = {A Family of Non-Isomorphic Maximal Function Fields},
  author = {Beelen, Peter and Montanucci, Maria and Niemann, Jonathan and Quoos, Luciane},
  year = 2025,
  journal = {Mathematische Zeitschrift},
  volume = {309},
  number = {2},
  pages = {19},
  issn = {0025-5874, 1432-1823},
  doi = {10.1007/s00209-024-03650-1}
}

@article{maAutomorphismGroups2016a,
  title = {On Automorphism Groups of Cyclotomic Function Fields over Finite Fields},
  author = {Ma, Liming and Xing, Chaoping and Yeo, Sze Ling},
  year = 2016,
  journal = {Journal of Number Theory},
  volume = {169},
  pages = {406--419},
  issn = {0022314X},
  doi = {10.1016/j.jnt.2016.05.026}
}

@article{montanucciAutomorphismGroup2024,
  title = {On the Automorphism Group of a Family of Maximal Curves Not Covered by the {{Hermitian}} Curve},
  author = {Montanucci, Maria and Tizziotti, Guilherme and Zini, Giovanni},
  year = 2024,
  journal = {Finite Fields and Their Applications},
  volume = {99},
  pages = {102498},
  issn = {10715797},
  doi = {10.1016/j.ffa.2024.102498}
}

@article{beelenWeierstrassSemigroups2023,
  title = {Weierstrass Semigroups and Automorphism Group of a Maximal Curve with the Third Largest Genus},
  author = {Beelen, Peter and Montanucci, Maria and Vicino, Lara},
  year = 2023,
  journal = {Finite Fields and Their Applications},
  volume = {92},
  pages = {102300},
  issn = {10715797},
  doi = {10.1016/j.ffa.2023.102300}
}

@article{beelenWeierstrassSemigroups2026a,
  title = {Weierstrass Semigroups and Automorphism Group of a Maximal Function Field with the Third Largest Possible Genus, $q \equiv 1 \pmod 3$},
  author = {Beelen, Peter and Montanucci, Maria and Vicino, Lara},
  year = 2026,
  journal = {Finite Fields and Their Applications},
  volume = {109},
  pages = {102701},
  issn = {10715797},
  doi = {10.1016/j.ffa.2025.102701}
}

@article{beelenWeierstrassSemigroups2026,
  title = {Weierstrass Semigroups and Automorphism Group of a Maximal Function Field with the Third Largest Possible Genus, $q \equiv 0 \pmod 3$},
  author = {Beelen, Peter and Montanucci, Maria and Vicino, Lara},
  year = 2026,
  journal = {Finite Fields and Their Applications},
  volume = {110},
  pages = {102729},
  issn = {10715797},
  doi = {10.1016/j.ffa.2025.102729}
}

@book{arbarelloGeometryAlgebraic1985,
  title = {Geometry of {{Algebraic Curves}}},
  author = {Arbarello, E. and Cornalba, M. and Griffiths, P. A. and Harris, J.},
  year = 1985,
  series = {Grundlehren Der Mathematischen {{Wissenschaften}}},
  volume = {267},
  publisher = {Springer New York},
  address = {New York, NY},
  doi = {10.1007/978-1-4757-5323-3},
  isbn = {978-1-4419-2825-2 978-1-4757-5323-3}
}

@article{carvalhoGoppaCodes2005,
  title = {On {{Goppa Codes}} and {{Weierstrass Gaps}} at {{Several Points}}},
  author = {Carvalho, C{\'i}cero and Torres, Fernando},
  year = 2005,
  journal = {Designs, Codes and Cryptography},
  volume = {35},
  number = {2},
  pages = {211--225},
  issn = {0925-1022, 1573-7586},
  doi = {10.1007/s10623-005-6403-4}
}

@incollection{matthewsWeierstrassSemigroups2009,
  title = {On {{Weierstrass Semigroups}} of {{Some Triples}} on {{Norm-Trace Curves}}},
  booktitle = {Coding and {{Cryptology}}},
  author = {Matthews, Gretchen L.},
  editor = {Chee, Yeow Meng and Li, Chao and Ling, San and Wang, Huaxiong and Xing, Chaoping},
  year = 2009,
  volume = {5557},
  pages = {146--156},
  publisher = {Springer Berlin Heidelberg},
  address = {Berlin, Heidelberg},
  doi = {10.1007/978-3-642-01877-0_13},
  isbn = {978-3-642-01813-8 978-3-642-01877-0}
}

@article{matthewsWeierstrassSemigroups2005,
  title = {Weierstrass {{Semigroups}} and {{Codes}} from a {{Quotient}} of the {{Hermitian Curve}}},
  author = {Matthews, Gretchen L.},
  year = 2005,
  journal = {Designs, Codes and Cryptography},
  volume = {37},
  number = {3},
  pages = {473--492},
  issn = {0925-1022, 1573-7586},
  doi = {10.1007/s10623-004-4038-5}
}

@incollection{matthewsMinimalGenerating2010,
  title = {Minimal Generating Sets of {{Weierstrass}} Semigroups of Certain $m$-Tuples on the Norm-Trace Function Field},
  booktitle = {Contemporary {{Mathematics}}},
  author = {Matthews, Gretchen L. and Peachey, Justin D.},
  editor = {McGuire, Gary and Mullen, Gary L. and Panario, Daniel and Shparlinski, Igor E.},
  year = 2010,
  volume = {518},
  pages = {315--326},
  publisher = {American Mathematical Society},
  address = {Providence, Rhode Island},
  doi = {10.1090/conm/518/10214},
  isbn = {978-0-8218-4786-2 978-0-8218-8197-2}
}

@article{huMultipointCodes2020,
  title = {Multi-Point Codes from the {{GGS}} Curves},
  author = {Hu, Chuangqiang and Yang, Shudi},
  year = 2020,
  journal = {Advances in Mathematics of Communications},
  volume = {14},
  number = {2},
  pages = {279--299},
  issn = {1930-5338},
  doi = {10.3934/amc.2020020}
}

@article{huMultipointCodes2018,
  title = {Multi-Point Codes over {{Kummer}} Extensions},
  author = {Hu, Chuangqiang and Yang, Shudi},
  year = 2018,
  journal = {Designs, Codes and Cryptography},
  volume = {86},
  number = {1},
  pages = {211--230},
  issn = {0925-1022, 1573-7586},
  doi = {10.1007/s10623-017-0335-7}
}

@article{yangWeierstrassSemigroups2017,
  title = {Weierstrass Semigroups from {{Kummer}} Extensions},
  author = {Yang, Shudi and Hu, Chuangqiang},
  year = 2017,
  journal = {Finite Fields and Their Applications},
  volume = {45},
  pages = {264--284},
  issn = {10715797},
  doi = {10.1016/j.ffa.2016.12.005}
}

@article{yangWeierstrassSemigroups2024,
  title = {Weierstrass Semigroups on the Third Function Field in a Tower Attaining the {{Drinfeld-Vl\u adu\c t}} Bound},
  author = {Yang, Shudi and Hu, Chuangqiang},
  year = 2024,
  journal = {Advances in Mathematics of Communications},
  volume = {18},
  number = {4},
  pages = {1051--1083},
  issn = {1930-5346, 1930-5338},
  doi = {10.3934/amc.2022066}
}

@article{castellanosGeneralizedWeierstrass2026,
  title = {On Generalized {{Weierstrass}} Semigroups in Arbitrary {{Kummer}} Extensions of $\mathbb{F}_q (x)$},
  author = {Castellanos, Alonso S. and Mendoza, Erik and Tizziotti, Guilherme},
  year = 2026,
  journal = {Finite Fields and Their Applications},
  volume = {112},
  pages = {102808},
  issn = {10715797},
  doi = {10.1016/j.ffa.2026.102808}
}

@article{tizziottiWeierstrassSemigroup2018,
  title = {Weierstrass {{Semigroup}} and {{Pure Gaps}} at {{Several Points}} on the {{GK Curve}}},
  author = {Tizziotti, G. and Castellanos, A. S.},
  year = 2018,
  journal = {Bulletin of the Brazilian Mathematical Society, New Series},
  volume = {49},
  number = {2},
  pages = {419--429},
  issn = {1678-7544, 1678-7714},
  doi = {10.1007/s00574-017-0059-3}
}

@article{maharajCodeConstruction2004,
  title = {Code {{Construction}} on {{Fiber Products}} of {{Kummer Covers}}},
  author = {Maharaj, H.},
  year = 2004,
  journal = {IEEE Transactions on Information Theory},
  volume = {50},
  number = {9},
  pages = {2169--2173},
  issn = {0018-9448},
  doi = {10.1109/TIT.2004.833356}
}

@article{castellanosWeierstrassSemigroup2020,
  title = {Weierstrass Semigroup at $m+1$ Rational Points in Maximal Curves Which Cannot Be Covered by the {{Hermitian}} Curve},
  author = {Castellanos, Alonso Sep{\'u}lveda and {Bras-Amor{\'o}s}, Maria},
  year = 2020,
  journal = {Designs, Codes and Cryptography},
  volume = {88},
  number = {8},
  pages = {1595--1616},
  issn = {0925-1022, 1573-7586},
  doi = {10.1007/s10623-020-00757-4}
}

@article{matthewsTriplesRational2021,
  title = {Triples of Rational Points on the {{Hermitian}} Curve and Their {{Weierstrass}} Semigroups},
  author = {Matthews, Gretchen L. and Skabelund, Dane and Wills, Michael},
  year = 2021,
  journal = {Journal of Pure and Applied Algebra},
  volume = {225},
  number = {8},
  pages = {106623},
  issn = {00224049},
  doi = {10.1016/j.jpaa.2020.106623}
}

@article{korchmarosHermitianCodes2013a,
  title = {Hermitian Codes from Higher Degree Places},
  author = {Korchm{\'a}ros, G. and Nagy, G.P.},
  year = 2013,
  journal = {Journal of Pure and Applied Algebra},
  volume = {217},
  number = {12},
  pages = {2371--2381},
  issn = {00224049},
  doi = {10.1016/j.jpaa.2013.04.002}
}

@article{geilBoundingNumber2009,
  title = {Bounding the Number of {{F}} q -Rational Places in Algebraic Function Fields Using {{Weierstrass}} Semigroups},
  author = {Geil, Olav and Matsumoto, Ryutaroh},
  year = 2009,
  journal = {Journal of Pure and Applied Algebra},
  volume = {213},
  number = {6},
  pages = {1152--1156},
  issn = {00224049},
  doi = {10.1016/j.jpaa.2008.11.013}
}

@misc{cotterillGapSets2025,
  title = {On Gap Sets in Arbitrary {{Kummer}} Extensions of ${{K}}(x)$},
  author = {Cotterill, Ethan and Mendoza, Erik A. R. and Speziali, Pietro},
  year = 2025,
  number = {arXiv:2506.19169},
  eprint = {2506.19169},
  primaryclass = {math},
  publisher = {arXiv},
  doi = {10.48550/arXiv.2506.19169},
  archiveprefix = {arXiv}
}
\end{document}